\documentclass[11pt]{article}
\usepackage[letterpaper,margin=1in]{geometry}
\usepackage[T1]{fontenc}
\usepackage{lmodern,amsmath,amssymb,amsthm,mathtools,mathrsfs,microtype}
\usepackage{enumitem,booktabs,array,aliascnt}
\usepackage{tikz,caption,needspace,placeins}
\usetikzlibrary{arrows.meta,calc,decorations.pathreplacing}
\usepackage[hidelinks]{hyperref}
\usepackage[nameinlink,capitalize]{cleveref}
\hypersetup{pdftitle={A sharp two-disk bound for the second positive Neumann eigenvalue under a curvature upper bound},pdfauthor={Meiqi Liu, Zhouyu Long, and Wenming Zou},pdfsubject={Neumann eigenvalues, curvature upper bounds, and sharp spectral comparison}}
\newtheorem{theorem}{Theorem}[section]
\newaliascnt{lemma}{theorem}\newtheorem{lemma}[lemma]{Lemma}\aliascntresetthe{lemma}
\crefname{lemma}{lemma}{lemmas}\Crefname{lemma}{Lemma}{Lemmas}
\newaliascnt{proposition}{theorem}\newtheorem{proposition}[proposition]{Proposition}\aliascntresetthe{proposition}
\crefname{proposition}{proposition}{propositions}\Crefname{proposition}{Proposition}{Propositions}
\newaliascnt{corollary}{theorem}\aliascntresetthe{corollary}
\crefname{corollary}{corollary}{corollaries}\Crefname{corollary}{Corollary}{Corollaries}
\theoremstyle{definition}
\newaliascnt{definition}{theorem}\aliascntresetthe{definition}
\newaliascnt{remark}{theorem}\newtheorem{remark}[remark]{Remark}\aliascntresetthe{remark}
\crefname{remark}{remark}{remarks}\Crefname{remark}{Remark}{Remarks}
\newcommand{\D}{\mathbb D}\newcommand{\R}{\mathbb R}\newcommand{\C}{\mathbb C}

\newcommand{\Span}{\operatorname{span}}

\numberwithin{equation}{section}
\title{\bf A sharp two-disk bound for the second positive Neumann
eigenvalue under a curvature upper bound}

\author{{\bf Meiqi Liu}\thanks{\raggedright School of Mathematical Sciences, Hangzhou Dianzi University, Hangzhou 310018, China. E-Mail: \href{mailto:liumq@hdu.edu.cn}{\nolinkurl{liumq@hdu.edu.cn}}},
\quad {\bf Zhouyu Long}\thanks{\raggedright Department of Mathematical Sciences, Tsinghua University, Beijing 100084, China.  E-Mail: \href{mailto:longzy25@mails.tsinghua.edu.cn}{\nolinkurl{longzy25@mails.tsinghua.edu.cn}}}, \quad  {\bf Wenming Zou}\thanks{\raggedright Department of Mathematical Sciences, Tsinghua University, Beijing 100084, China.  E-Mail:  \href{mailto:zou-wm@mail.tsinghua.edu.cn}{\nolinkurl{zou-wm@mail.tsinghua.edu.cn}} }}

\date{ }

\begin{document}
\maketitle

\begin{abstract}
We prove the sharp two-disk comparison conjectured by Langford and Laugesen
for the second positive Neumann eigenvalue under an upper Gaussian-curvature
bound, and strengthen it to a shifted reciprocal inequality. Let
$\Omega\subset\mathbb C$ be bounded, simply connected and Lipschitz, with
$g=\omega|dz|^2$, where $\omega$ is continuous and positive on
$\overline\Omega$ and $\Delta\log\omega+2K\omega\ge0$ in distributions.
Write $A=|\Omega|_g$ and enumerate the spectrum by
$0=\lambda_0<\lambda_1\le\lambda_2\le\cdots$. If $KA<4\pi$ when $K>0$, then
\[
 \frac1{\lambda_2}+\frac1{\lambda_3}>
 \frac2{\lambda_1(D_K(A/2))},\qquad
 \lambda_2<\lambda_1(D_K(A/2)),
\]
where $D_K(A/2)$ is the constant-curvature disk of area $A/2$.
The same conclusions hold for disk-type Lipschitz domains in smooth surfaces.
In the subcritical regime, both bounds are sharp, as witnessed by an
extremizing sequence of smooth
connected domains of exact area and curvature whose Neumann eigenvalues
converge, at every fixed index, to those of two disjoint equal model disks.
At $KA=4\pi$ we obtain the sharp non-strict bounds
$\lambda_2\le2K$ and $\lambda_2^{-1}+\lambda_3^{-1}\ge K^{-1}$.
The proof combines two-pole Green coordinates, positive-kernel comparison,
and simultaneous centering. We also establish a global quadratic deficit
for the model area split.

\vskip0.1in
\noindent\text{\bf Mathematics Subject Classification:} 35P15; 58J50
\end{abstract}

\section{Introduction and main results}\label{sec:scope}
Sharp Neumann eigenvalue comparisons identify different extremal geometries for the first and second positive eigenvalues. Szeg\H{o}~\cite{Szego} and Weinberger~\cite{Weinberger} established the disk and ball comparison for the first positive eigenvalue. Girouard--Nadirashvili--Polterovich~\cite{GNP} proved the two-disk bound for the second positive eigenvalue on simply connected planar domains; see also Girouard--Polterovich~\cite{GP} for strictness and conformal-density extensions. Bucur--Henrot~\cite{BH} obtained the Euclidean two-ball comparison in every dimension, and Freitas--Laugesen~\cite{FL} and Bucur--Martinet--Nahon~\cite{BMN} proved its hyperbolic and spherical counterparts.

For surfaces with a Gaussian curvature upper bound, Bandle~\cite{Bandle1972} proved the same-area first-positive disk comparison, with $KA\le2\pi$ when $K>0$. Langford--Laugesen~\cite[Theorem~1.1]{LL} extended the positive-curvature range beyond a hemisphere, and Provenzano--Savo~\cite[Theorem~1.3]{PS} obtained the comparison throughout $KA\le4\pi$. Langford--Laugesen~\cite[Conjecture~1.4]{LL} asked for the corresponding sharp two-disk bound for the second positive eigenvalue when $KA<4\pi$. In their indexing $0=\mu_1<\mu_2\le\mu_3\le\cdots$, the conjecture is
\[
 \mu_3(\Omega)<\mu_3(D_K\sqcup D_K)=\mu_2(D_K),
\]
where each model disk has area $|\Omega|/2$ and the extremal value is approached by a degeneration to their disjoint union. The problem also appears in Laugesen's Oberwolfach contribution~\cite[p.~2099]{OWR}.

We prove this conjecture for bounded Lipschitz membranes, allowing continuous positive weights with a distributional curvature upper bound, and for disk-type Lipschitz domains in smooth surfaces. This includes and weakens the weight assumptions of~\cite{LL}. We obtain the stronger strict inequality for the shifted reciprocal sum $\lambda_2^{-1}+\lambda_3^{-1}$, and construct an extremizing sequence of smooth connected domains with exact prescribed area and curvature. The construction gives convergence of every fixed-index Neumann eigenvalue to the spectrum of two equal disks.

The proof must compare two Green branches whose profiles share an interface
trace, and select the poles so that the optimized field is orthogonal to
both the constant mode and a first positive eigenfunction. Coefficient order
alone does not compare the profile energies; a positive-kernel formulation
supplies the comparison. Compatibility across the interior zero set gives
strictness, while continuity of the optimized fields permits the
simultaneous moment selection.

\paragraph{Two settings for the same comparison.}
In the \emph{weighted setting (L)}, let $\Omega\subset\mathbb C$ be a bounded
simply connected domain with Lipschitz boundary, and let
\begin{equation}
 \omega\in C(\overline\Omega),\qquad
 \omega>0\ \hbox{on }\overline\Omega,\qquad g=\omega|dz|^2.
 \label{eq:LL-weight}
\end{equation}
The curvature upper bound is understood through
\begin{equation}
 \Delta\log\omega+2K\omega\ge0\quad\hbox{in }\mathcal D'(\Omega),
 \label{eq:LL-curvature}
\end{equation}
that is,
\[
 \int_\Omega\log\omega\,\Delta\varphi\,dx
   +2K\int_\Omega\omega\varphi\,dx\ge0
 \qquad(0\le\varphi\in C_c^\infty(\Omega)).
\]
When $\omega\in C^2(\Omega)$ this is exactly the classical inequality
$K_g=-\Delta\log\omega/(2\omega)\le K$.  Thus the assumptions include the
weight regularity of~\cite[p.~2256]{LL}, and require neither an interior
classical curvature function nor boundary derivatives of the weight.
In setting (L), the notation $K_g\le K$ below means
\eqref{eq:LL-curvature}; in setting (G) it has its usual pointwise meaning.
The Laplacian and Neumann spectrum in setting (L) are defined by the forms
\eqref{eq:mass-energy-forms}.

In the \emph{smooth geometric setting (G)}, let $\Omega\Subset(S,g)$ be a connected Lipschitz domain whose closure is homeomorphic to a closed disk, where $g$ is smooth and positive definite on a neighborhood of $\overline\Omega$.

In either setting the Neumann form has domain $H^1(\Omega)$ and is
\begin{equation}
 E(f,h)=\int_\Omega\langle\nabla f,\nabla h\rangle_g\,dv_g,
 \qquad M(f,h)=\int_\Omega fh\,dv_g.
 \label{eq:mass-energy-forms}
\end{equation}
For (L), $dv_g=\omega\,dx$ and $E(f,h)=\int_\Omega\nabla f\cdot\nabla h\,dx$. All $L^2$ inner products and orthogonality conditions on $\Omega$ use $dv_g$ unless another measure is specified. Write $E(f)=E(f,f)$ and $M(f)=M(f,f)$; for a complex function these are the sums for its two real components. The eigenvalues
\begin{equation}
 0=\lambda_0<\lambda_1\le\lambda_2\le\cdots
 \label{eq:indexing}
\end{equation}
are counted with multiplicities. For a disconnected domain, we retain the same
zero-based enumeration, counting all zero eigenvalues; the inequality
$\lambda_0<\lambda_1$ then need not hold.
In the complete simply connected model surface $\mathbb M_K$, write
\begin{equation}
 D_K(s)=\text{the geodesic disk of area }s,
 \qquad \beta_K(s)=\lambda_1(D_K(s)),
 \label{eq:modelnotation}
\end{equation}
with $0<s<4\pi/K$ when $K>0$.

\begin{theorem}\label{thm:main}
In either setting (L) or (G), suppose
\begin{equation}
 K_g\le K\quad\hbox{on }\Omega,
 \qquad A=|\Omega|_g>0,
 \qquad KA<4\pi\quad\hbox{if }K>0.
 \label{eq:assumptions}
\end{equation}
Then
\begin{equation}
 \frac1{\lambda_2(\Omega,g)}+\frac1{\lambda_3(\Omega,g)}
 >\frac{2}{\beta_K(A/2)}.
 \label{eq:shifted-main}
\end{equation}
Consequently,
\begin{equation}
 \lambda_2(\Omega,g)<\beta_K(A/2).
 \label{eq:main}
\end{equation}
No simplicity assumption is imposed on $\lambda_1$. In setting (L), the curvature assumption is \eqref{eq:LL-curvature}; no derivative of $\omega$ is required.
\end{theorem}

\begin{theorem}\label{thm:sharp}
For every real $K$ and every $A>0$ satisfying the area condition in \eqref{eq:assumptions}, there exist smooth connected domains $U_n\subset\mathbb M_K$, each with disk-type closure, such that
\begin{equation}
 |U_n|=A,\qquad K_{U_n}\equiv K,
 \qquad
 \lambda_k(U_n)\longrightarrow
 \lambda_k\bigl(D_K(A/2)\sqcup D_K(A/2)\bigr)
 \quad\text{for every fixed }k\ge0.
 \label{eq:sharpfull}
\end{equation}
In particular,
\begin{equation}
 \lambda_1(U_n)\longrightarrow0,
 \qquad\lambda_2(U_n)\longrightarrow\beta_K(A/2).
 \label{eq:sharp}
\end{equation}
Consequently, in each of settings (L) and (G) with fixed $K,A$ (allowing the ambient surface to vary in (G)),
\begin{equation}
 \sup\lambda_2=\beta_K(A/2),\qquad
 \inf\left(\frac1{\lambda_2}+\frac1{\lambda_3}\right)
   =\frac{2}{\beta_K(A/2)}.
 \label{eq:sharp-both}
\end{equation}
Neither extremal value is attained by a connected member of either class.
\end{theorem}

\begin{theorem}\label{thm:critical}
In either setting \textup{(L)} or \textup{(G)}, suppose $K>0$,
$K_g\le K$, and $A=4\pi/K$. Then
\begin{equation}
 \lambda_2(\Omega,g)\le2K,
 \qquad
 \frac1{\lambda_2(\Omega,g)}+\frac1{\lambda_3(\Omega,g)}
 \ge\frac1K.
 \label{eq:critical-comparisons}
\end{equation}
Both constants are sharp at fixed $K,A$ in each class (allowing the ambient
surface to vary in setting \textup{(G)}): there exist
smooth connected disk-type domains, of area $A$ and with constant
curvatures $K_n<K$ tending to $K$, for which
\[
 \lambda_1\longrightarrow0,
 \qquad \lambda_2,\lambda_3\longrightarrow2K.
\]
Consequently the scalar supremum is $2K$ and the reciprocal infimum
is $1/K$. No equality characterization at the critical area is
asserted.
\end{theorem}

\begin{remark}
With our indexing, $\mu_3$ in~\cite{LL} is $\lambda_2$ and $\mu_2(D_K)$ is $\beta_K(A/2)$. Theorem~\ref{thm:main} proves \cite[Conjecture~1.4]{LL}, with a strict reciprocal strengthening and weaker weight regularity.

Theorem~\ref{thm:sharp} constructs the predicted degeneration to two equal disks. The area condition $KA<4\pi$ is automatic for $K\le0$. Theorem~\ref{thm:critical} treats the critical positive-curvature area.
\end{remark} 

\paragraph{Method and related results.}
The one-pole Green-level approach of Provenzano--Savo~\cite[Sections~1.1 and 7]{PS} optimizes a radial profile and selects a pole to remove the constant mode. Reciprocal estimates originate in Hersch's trace variational principle~\cite{Hersch1961}; see Hile--Xu~\cite{HileXu1993}. For the two-ball problem, Chen--Yang~\cite{CY} proved the shifted reciprocal comparison on space forms asked for in~\cite[Remark~11]{BMN}. Their recent work~\cite[Corollary~3.2]{CYHM} proves the same-area harmonic-mean comparison $\lambda_1^{-1}+\lambda_2^{-1}\ge2/\beta_K(A)$ for compact simply connected surfaces with nonempty smooth boundary and $K_g\le K$, assuming $KA\le4\pi$ when $K>0$; at the critical area the model is the full sphere. It also establishes magnetic Green-level comparisons. Our shifted pair $\lambda_2,\lambda_3$ is compared with the half-area disk.

Our proof uses the signed two-pole Green function $H=\psi_p-\psi_q$. Its sign branches share an interface trace, and the optimized profile is governed by the positive kernel $e^{-2\pi|t-s|}/(4\pi)$. Contraction of Green-time distances compares this coupled problem with two model branches whose areas may differ. Fixed-spectral-parameter concavity for $K\ge0$ and cumulative-area concavity for $K<0$ give the equal-area maximum, and compatibility across the interior zero curve excludes equality below the critical area $KA=4\pi$.

The profile depends on the pole pair. We prove continuity up to boundary and coalescing parameters and use the boundary-face identities to compute degree $-2$. This simultaneously removes the constant mode and one fixed first positive eigenfunction, including when that eigenvalue is multiple. The real components have isotropic Dirichlet energy, so the shifted reciprocal variational estimate applies. The moment selection extends the normalization principles of Hersch~\cite{Hersch1970} and the disk constructions of~\cite{GNP,GP,LaugesenHS} to the optimized profiles. Section~\ref{sec:quadratic} further quantifies spectral, area-split, and kernel deficits of the selected configuration; related geometric stability results appear in~\cite{Nadirashvili1997,BrascoPratelli2012,LiWang2026}.

\begin{remark}
The two-pole kernel comparison avoids a positive-curvature obstruction in
conformal folding. For smooth positive densities $h_i$ of curvature $K_i$,
write $h=h_1+h_2$ and $p_i=h_i/h$. Direct differentiation gives
\[
 -\frac{\Delta\log h}{2h}
 =p_1^2K_1+p_2^2K_2
 -\frac{p_1p_2}{2h}|\nabla\log(h_1/h_2)|^2.
\]
The assumptions $K_i\le K\le0$ imply a curvature bound $K/2$ for the sum.
For $K>0$ this conclusion fails: the densities
$\rho_a(r)=a/(1+Kar^2/4)^2$ each have curvature $K$, but
\[
 -\frac{\Delta\log(\rho_a+\rho_b)}{2(\rho_a+\rho_b)}(0)
 =\frac K2+\frac{K(a-b)^2}{2(a+b)^2}>\frac K2\quad(a\ne b).
\]
\end{remark}

\paragraph{Scope of the main results.}\label{sec:combinedscope}
The disk-type assumption is essential for the general curvature-bound statement. On the flat cylinder $(0,1)\times(\mathbb R/\mathbb Z)$, with Neumann conditions at its two boundary circles, $A=1$, $K_g=0$, and separation of variables gives $\lambda_1=\pi^2$, $\lambda_2=\lambda_3=4\pi^2$. Write $b_{\rm E}=\lambda_1(\D,\mathrm{flat})$. The unit-disk test $u(x,y)=x$ gives $b_{\rm E}\le4$, so $\beta_0(1/2)=2\pi b_{\rm E}\le8\pi<4\pi^2$. Thus the two inequalities of Theorem~\ref{thm:main} would fail if the topology restriction were simply removed.
The reciprocal inequality does not separately bound $\lambda_3$ by $\beta_K(A/2)$. In (L), the curvature upper bound is the distributional inequality \eqref{eq:LL-curvature}; the sign cases refer to the parameter $K$. The weights remain continuous and strictly positive on the compact closure. At $KA=4\pi$ the interface obstruction vanishes, and Theorem~\ref{thm:critical} gives non-strict comparisons without an equality classification.

\paragraph{Organization.}
Sections~\ref{sec:prelimmodel} and~\ref{sec:green} develop the analytic and geometric preliminaries, the two-pole trial functions, and the positive-kernel comparison. Section~\ref{sec:model} proves the model maximum and strict domain comparison. Section~\ref{sec:center} establishes the simultaneous moment conditions, treats the interior two-pole and one-pole boundary-parameter alternatives, and completes the spectral argument.

Section~\ref{sec:sharpgeometry} proves sharpness, and Section~\ref{sec:quadratic} gives the quantitative estimates. Appendices~\ref{sec:cyclic} and~\ref{app:regularity} contain stronger bounds under cyclic symmetry and an interface-regularity example.

\section{Analytic and geometric preliminaries}\label{sec:prelimmodel}
We establish interior disk coordinates, the local Bol--Fiala inequality
under the weak curvature bound, and the half-area model profile. The
curvature inequality will be applied to relatively compact Green
superlevel disks.
\subsection{Forms, conformal coordinates, and local curvature}
\paragraph{Conventions.}
For smooth metrics we use $\Delta_g=\operatorname{div}_g\nabla$; in setting (L), Laplacians are understood weakly through the energy form. Thus $-\Delta_g\psi_p=\delta_p$ in the Green-function sense. Unless specified otherwise, geometric area and length use $dv_g$ and $ds_g$. Set
\begin{equation}
 A=|\Omega|_g,\qquad m=A/2,\qquad
 \beta=\beta_K(m),\qquad c=2\pi.
 \label{eq:notation}
\end{equation}

\begin{lemma}\label{lem:LL-form}
In setting (L), the form \eqref{eq:mass-energy-forms} is densely defined, closed and nonnegative in $L^2(\Omega,\omega\,dx)$, with form domain $H^1(\Omega)$. Its form-domain embedding into this $L^2$ space is compact. Its kernel consists exactly of the constants, and its eigenvalues satisfy \eqref{eq:indexing}. The associated Neumann operator has compact resolvent. The same statements hold in setting (G).
\end{lemma}
\begin{proof}
In (L), positivity and continuity on the compact closure give
$0<\omega_-\le\omega\le\omega_+<\infty$. Thus the weighted and unweighted
$L^2$ norms are equivalent, and $E(f)+M(f)$ is equivalent to the squared
$H^1(\Omega)$ norm. Rellich compactness on the fixed Lipschitz domain gives
the compact form-domain embedding~\cite[Chapter~3]{McLean}. Connectedness
makes the kernel precisely the constants. The representation theorem for
closed quadratic forms, compactness of the form-domain embedding, and the
min--max principle give the stated spectrum~\cite[Sections~2.3 and 4.3]{Teschl},
with weak eigenvalue equation
\[
 \int_\Omega\nabla f\cdot\nabla h\,dx
 =\lambda\int_\Omega fh\omega\,dx\qquad(h\in H^1(\Omega)).
\]
In (G), a finite metric atlas on the compact closure gives the same norm
comparability and compactness. No boundary condition is imposed on the form
domain; the Neumann boundary condition is encoded in the associated operator.
\end{proof}

\paragraph{Interior conformal coordinates.}
In (L), a Riemann map $\Phi:\D\to\Omega$ is an interior biholomorphism. For (G), the disk-type domain is orientable and simply connected, so uniformization~\cite[Section~27]{Forster} leaves the disk and plane as the two noncompact conformal types. To exclude the plane, choose a smooth ambient function $f$ with $0\le f\le1$ and nonconstant boundary trace. On the fixed Lipschitz domain, the Poincar\'e inequality on $H^1_0(\Omega)$ and the direct method give the unique minimizer of the Dirichlet energy in $f+H^1_0(\Omega)$. It is weakly harmonic. Truncation to $[0,1]$ preserves its trace and does not increase the energy, so uniqueness gives $0\le u\le1$. Its nonconstant trace makes $u$ nonconstant, and interior elliptic regularity makes it smooth. If $\Omega$ had the planar conformal type, pulling $u$ back to $\mathbb C$ would contradict Liouville's theorem for bounded harmonic functions. Thus (G) also has an interior conformal map $\Phi:\D\to\Omega$.

In either setting the pullback measure has a positive continuous interior
density and finite mass.  No boundary bounds or extension of this density to
$\partial\D$ are assumed.  In (L) it is
\begin{equation}
 \rho(z)=\omega(\Phi(z))|\Phi'(z)|^2,
 \qquad \int_\D\rho\,dx=A.
 \label{eq:LL-pullback}
\end{equation}
The distributional curvature inequality is invariant under these conformal
coordinates.  Indeed, if $\Psi:V\to U$ is a conformal diffeomorphism,
$h>0$ is continuous on $U$, and
$\widetilde h=(h\circ\Psi)|\Psi'|^2$, then for
$0\le\varphi\in C_c^\infty(V)$ the harmonicity of
$\log|\Psi'|^2$ and the conformal Laplacian identity give
\[
 \int_V\log\widetilde h\,\Delta\varphi\,dx
  =\int_U\log h\,\Delta(\varphi\circ\Psi^{-1})\,dx,
 \qquad
 \int_V\widetilde h\varphi\,dx
  =\int_Uh(\varphi\circ\Psi^{-1})\,dx.
\]
Here the transported test function is extended by zero outside its compact
support.  Thus $\Delta\log h+2Kh\ge0$ on $U$ implies the same inequality for
$\widetilde h$ on $V$.  This also applies to every branch coordinate below.

\begin{lemma}\label{lem:local-bol}
Let $U\subset\mathbb C$ be open and let $g=h|dz|^2$, where $h>0$ is
continuous on $U$ and
\[
 \Delta\log h+2Kh\ge0\quad\hbox{in }\mathcal D'(U).
\]
For every smooth disk-type domain $V\Subset U$,
\begin{equation}
 L_g(\partial V)^2\ge |V|_g(4\pi-K|V|_g).
 \label{eq:bol}
\end{equation}
The conclusion applies in particular to the relatively compact Green
superlevel disks in settings (L) and (G).
\end{lemma}
\begin{proof}
If $K>0$ and $K|V|_g\ge4\pi$, the right-hand side is nonpositive.  Otherwise
$4\pi-K|V|_g>0$.  Put $w=\tfrac12\log h$, choose fixed neighborhoods
$\overline V\subset U_0\Subset U_1\Subset U$, and let $\zeta_n$ be standard
nonnegative mollifiers with supports tending to zero.  All convolutions
below are taken in these interior neighborhoods.  The functions
$w_n=w*\zeta_n$ are smooth near $\overline U_0$, converge uniformly there to
$w$, and satisfy
\[
 \Delta w_n\ge-K(h*\zeta_n).
\]
For $h_n=e^{2w_n}$ and $g_n=h_n|dz|^2$, it follows that
\[
 K_{g_n}=-h_n^{-1}\Delta w_n
       \le K\frac{h*\zeta_n}{h_n}\le K+\varepsilon_n
       \quad\hbox{on }\overline V,
\]
where one may take
$\varepsilon_n=|K|\sup_{\overline V}|(h*\zeta_n)/h_n-1|\to0$.
This estimate holds for either sign of $K$.  Uniform convergence of the
metric factors gives
$|V|_{g_n}\to|V|_g$ and $L_{g_n}(\partial V)\to L_g(\partial V)$; in
particular $(K+\varepsilon_n)|V|_{g_n}<4\pi$ for all sufficiently large $n$.

To use the smooth Bol--Fiala inequality, choose
$\eta\in C_c^\infty(U_0)$ equal to one near $\overline V$ and extend
$\eta w_n$ by zero to $\mathbb C$.  The metric
$e^{2\eta w_n}|dz|^2$ is complete, agrees with $g_n$ near $\overline V$,
and satisfies the required curvature bound on $V$.  Apply the classical
inequality~\cite[Isoperimetric inequality (I)]{CF}, \cite{Bandle}, with bound
$K+\varepsilon_n$, and pass to the limit.
\end{proof}

\subsection{Model disks and radial factors of Neumann eigenfunctions}
Define
\begin{equation}
 S_K(r)=\begin{cases}
 \sin(\sqrt K r)/\sqrt K,&K>0,\\
 r,&K=0,\\
 \sinh(\sqrt{-K}r)/\sqrt{-K},&K<0,
 \end{cases}
 \qquad V_K(r)=2\pi\int_0^rS_K(t)\,dt.
 \label{eq:polar}
\end{equation}
The model metric is $dr^2+S_K^2d\theta^2$. For $K>0$ its disk radius is less than $\pi/\sqrt K$. With $a=V_K(r)$,
\begin{equation}
 S_K''=-KS_K,\quad S_K'^2+KS_K^2=1,\quad
 S_K'=1-Ka/c,\quad
 G_K(a):=(cS_K)^2=a(2c-Ka).
 \label{eq:Gmodel}
\end{equation}

\begin{lemma}\label{lem:angularmodel}
On a model disk of area $s=V_K(R)$, the first positive Neumann eigenspace is exactly
$\Span\{v(r)\cos\theta,v(r)\sin\theta\}$, where
\begin{equation}
 -(S_Kv')'+v/S_K=\beta_K(s)S_Kv,\qquad
 v(0)=0,\quad v'(R)=0,\quad v(R)=1.
 \label{eq:radialmodel}
\end{equation}
The factor is positive on $(0,R]$ and satisfies $v(r)=b r+O(r^3)$, $b>0$. At $s=m$ under \eqref{eq:assumptions}, $v'>0$ on $(0,R)$, and $\beta>2K$ when $K\ge0$.
\end{lemma}
\begin{proof}
\emph{Angular sectors.} In angular sector $\ell\ge1$ the quotient is
\[
 \frac{\int_0^R(S_K f'^2+\ell^2 f^2/S_K)\,dr}
      {\int_0^R S_K f^2\,dr}.
\]
All these sectors have the same radial form domain. For $\ell\ge2$ the quotient exceeds its $\ell=1$ counterpart by at least $(\ell^2-1)/\max S_K^2$. The sector-$1$ minimum is attained by compactness. Taking the radial absolute value and using the ODE gives positivity, including at the Neumann endpoint. Regular solutions have the stated center expansion and form a one-dimensional space: for two such solutions at the same spectral parameter,
$S_K(v_1v_2'-v_1'v_2)$ is constant and tends to zero at the center. Its vanishing proves linear dependence.

\emph{Radial modes.} To exclude a lower positive radial mode, differentiate its equation
$q''+(S_K'/S_K)q'+\lambda q=0$. Since $(S_K'/S_K)'=-S_K^{-2}$, the nonzero function $q'$ is a sector-$1$ eigenfunction with the Dirichlet condition at $r=R$. The center expansion $q'(r)=-\lambda q(0)r/2+O(r^3)$ and $q''=O(1)$ ensure that $q'$ belongs to that angular form domain. That lowest Dirichlet value strictly exceeds the sector-$1$ Neumann minimum: equality would make an attained Dirichlet minimizer also satisfy the Neumann condition, and the two zero endpoint data force the solution to vanish. Thus the first positive eigenspace has exactly the two stated angular modes.

\emph{The half-area disk.} At the half-area scale $S_K$ is strictly increasing (the spherical radius is below a hemisphere). Set $y=S_Kv'$. Then
$y'=(1-\beta S_K^2)v/S_K$, whose coefficient changes sign at most once, from positive to negative. Since $y(0)=y(R)=0$, it must change sign and $y>0$ in between.

To compare the eigenvalue with $2K$, use $h=S_K$, which solves the first-angular equation with spectral parameter $2K$. Integration by parts after writing $v=h(v/h)$ yields
\begin{equation}
 (\beta-2K)\int_0^R S_Kv^2\,dr
 =S_K'(R)v(R)^2+\int_0^R S_K^3((v/S_K)')^2\,dr>0
 \label{eq:picone}
\end{equation}
when $K\ge0$. The center term vanishes and $S_K'(R)>0$.
\end{proof}

In area coordinates, the angular equation is
\begin{equation}
 -(G_K f')'+\frac{c^2}{G_K}f=\beta_K(s)f,
 \qquad f(a)=O(\sqrt a),\quad f'(s)=0.
 \label{eq:areaeigenmodel}
\end{equation}
The model Green-time coordinate, normalized by $T_{K,s}(s)=0$, and its inverse are
\begin{align}
 T_{K,s}(a)&=\int_a^s\frac{d\alpha}{G_K(\alpha)}
 =\frac1{4\pi}\log\frac{s(4\pi-Ka)}{a(4\pi-Ks)},\label{eq:modelclock}\\
 a_{K,s}(t)&=\frac{4\pi s\vartheta(t)}{4\pi-Ks(1-\vartheta(t))},
 \qquad \vartheta(t)=e^{-4\pi t}.\label{eq:model-arbitrary-superlevel}
\end{align}
In the conformal unit-disk radius, the model density of curvature
$\varkappa$ and total area $s$ is
\begin{equation}
 q_{\varkappa,s}(r)=\frac{4s(4\pi-\varkappa s)}
 {[4\pi-\varkappa s+\varkappa sr^2]^2}.
 \label{eq:curvature-density}
\end{equation}
Here $s>0$ and $\varkappa s<4\pi$ if $\varkappa>0$.
Direct differentiation gives
$\Delta\log q_{\varkappa,s}+2\varkappa q_{\varkappa,s}=0$,
and its cumulative area is
\begin{equation}
 \mathcal A_{\varkappa,s}(r)=2\pi\int_0^r t q_{\varkappa,s}(t)\,dt
 =\frac{4\pi sr^2}{4\pi-\varkappa s+\varkappa sr^2}.
 \label{eq:Acum-model}
\end{equation}
For $\varkappa=K$ and $r=e^{-ct}$, this is
\eqref{eq:model-arbitrary-superlevel}.
Write $a_0=a_{K,m}$, $B_0=-a_0'=G_K(a_0)$, and express the half-area profile in the Green-time coordinate as $g_0(t)$. Then
\begin{gather}
 -g_0''+c^2g_0=\beta B_0g_0,\quad
 g_0(0)=1,\quad g_0'(0)=0,\quad
 w_0:=-(g_0^2)'>0\ (t>0),\label{eq:greentimeprofile}\\
 M_0=\int_0^\infty g_0^2B_0\,dt=\int_0^m f_0^2\,da,
 \qquad E_0=\int_0^\infty(g_0'^2+c^2g_0^2)\,dt=\beta M_0.
 \label{eq:modelenergy}
\end{gather}
Here $f_0(a)=g_0(T_{K,m}(a))$. The center expansion gives $g_0,g_0'=O(e^{-ct})$, justifying integration by parts, and $\int_0^\infty w_0=1$. The constants $E_0,M_0$ are for one \emph{complex} angular field; a real component has half of each.

\section{Two-pole Green functions and kernel comparison}\label{sec:green}
In either setting, fix the interior conformal map $\Phi:\D\to\Omega$ and set $d\mu=\Phi^*dv_g$. It has mass $A$ and a positive continuous interior density. For $p=\Phi(a)\ne q=\Phi(b)$, let $\psi_p,\psi_q$ have zero Dirichlet boundary data and the same logarithmic coefficient $(2\pi)^{-1}$, and define
\begin{equation}
 B_a(z)=\frac{z-a}{1-\bar a z},\qquad
 \Xi=\frac{B_a}{B_b},\qquad
 H=\psi_p-\psi_q=-c^{-1}\log|\Xi|,\qquad
 \chi=\frac{\Xi}{|\Xi|}.
 \label{eq:green}
\end{equation}
The last two expressions are understood after pullback. The meromorphic quotient $\Xi$ has one zero and one pole. Values at the poles do not affect area integrals, since the area measure is absolutely continuous.

We first verify the Dirichlet Green-function identity on the physical domain. On the unit disk, set
\[
 G_{\D,a}(z)=-c^{-1}\log|B_a(z)|.
\]
Choose $\zeta_a\in C_c^\infty(\D)$ equal to one near $a$ and put
$\sigma_a(z)=-c^{-1}\zeta_a(z)\log|z-a|$. The logarithmic singularities cancel in
$r_a:=G_{\D,a}-\sigma_a$; the explicit formula for $B_a$ shows that $r_a$ is smooth across $a$, is smooth up to $\partial\D$, and vanishes there. Hence $r_a\in H^1_0(\D)$. Since $r_a$ and its gradient are bounded and $r_a$ vanishes on
$\partial\D$, boundary cutoffs followed by interior mollification give
$r_{a,n}\in C_c^\infty(\D)$ converging uniformly and in Dirichlet energy to
$r_a$: on a layer of width $\varepsilon$, one has $r_a=O(\varepsilon)$,
so the cutoff-gradient energy is $O(\varepsilon)$.
Composition with $\Phi^{-1}$ preserves Dirichlet energy; uniform convergence
and the finite area give convergence in physical $L^2$. Therefore
$r_a\circ\Phi^{-1}\in H^1_0(\Omega)$.
Consequently
\[
 \Psi:=G_{\D,a}\circ\Phi^{-1}
 =\sigma_a\circ\Phi^{-1}+r_a\circ\Phi^{-1}
\]
has the defining singularity and boundary condition of the Dirichlet Green function: after subtracting the cutoff logarithmic singularity written in the local conformal coordinate $z=\Phi^{-1}(x)$ near $p=\Phi(a)$, the remainder lies in $H^1_0(\Omega)$. Conformal energy invariance also gives, for every $\varphi\in C_c^\infty(\Omega)$,
\begin{equation}
 \int_\Omega\langle\nabla\Psi,\nabla\varphi\rangle_g\,dv_g
 =\int_\D\nabla G_{\D,a}\cdot\nabla(\varphi\circ\Phi)\,dx\,dy
 =\varphi(\Phi(a)).
 \label{eq:green-dirichlet-identification}
\end{equation}
The integrals are locally finite at the logarithmic pole. These two properties characterize the Dirichlet Green function: if two such representatives have the same source, their difference belongs to $H^1_0(\Omega)$ and is weakly harmonic, so testing by the difference gives zero Dirichlet energy. Thus $\Psi=\psi_p$. The $H^1_0$ condition applies to the regular remainder, not to the logarithmic Green function.

\begin{lemma}\label{lem:green}
The sign domains $\Omega_\pm$ are conformal disks, on which $h_+=H$ and $h_-=-H$ are their respective positive Dirichlet Green functions. The maps $\Xi$ and $1/\Xi$ are conformal disk coordinates on the respective branches. Every positive level is regular and compactly contained in $\Omega$, and
\begin{equation}
 \int_{h_j=t}|\nabla h_j|\,ds_g=1\quad(t>0).
 \label{eq:flux}
\end{equation}
The interior zero set $\Gamma$ is nonempty and smooth, with $\nabla H\ne0$.
\end{lemma}
\begin{proof}
A disk automorphism sends the poles to $-r,r$, $0<r<1$, changing the quotient only by a unimodular constant. This normalization is used only for the geometry; the parameter family retains the phase fixed in \eqref{eq:green}. In that coordinate,
\[
 (\log(B_{-r}/B_r))'
 =-\frac{2r(1-r^2)(1+z^2)}{(z^2-r^2)(1-r^2z^2)}.
\]
Its only zeros are the boundary points $\pm i$. For $z=x+iy$, subtraction gives
\[
 |B_{-r}|^2-|B_r|^2=
 \frac{4rx(r^2-1)(|z|^2-1)}
 {(1-2rx+r^2|z|^2)(1+2rx+r^2|z|^2)}.
\]
The denominator is positive, so the zero set is the imaginary diameter and the two branches are half-disks. On the positive branch, $\Xi$ is a proper holomorphic map onto $\D$, has modulus one on both boundary arcs and exactly one simple zero. Its degree is one, hence it is conformal onto $\D$; the negative branch uses $1/\Xi$. Thus, with branch coordinate $\zeta=\Xi$ on $\Omega_+$ and $\zeta=1/\Xi$ on $\Omega_-$, one has $h_j=-c^{-1}\log|\zeta|$ and $\{h_j>t\}=\{|\zeta|<e^{-ct}\}$. Since the product $|\nabla h_j|_g\,ds_g$ is unchanged by a conformal change of metric in dimension two, on the level $|\zeta|=r$,
\[
 \int_{h_j=t}|\nabla h_j|_g\,ds_g
 =\int_0^{2\pi}\frac1{cr}r\,d\theta=\frac{2\pi}{c}=1.
\]
All these level-set computations take place in the interior.
\end{proof}

Write $s_j=|\Omega_j|$; then $s_++s_-=A$. Figure~\ref{fig:green-branches} illustrates the coordinate geometry and distinguishes the two types of branch-boundary arcs.

\begin{figure}[htbp]
\centering
\begin{tikzpicture}[x=0.94cm,y=0.94cm,font=\small,>=Stealth]
  \def\rad{1.12}
  \pgfmathsetmacro{\th}{acos(-7/25)}
  \begin{scope}[shift={(2.05,0)}]
    \draw[semithick] (-\th:\rad) arc[start angle=-\th,end angle=\th,radius=\rad];
    \draw[semithick,dashed] (\th:\rad) arc[start angle=\th,end angle=360-\th,radius=\rad];
    \foreach \r in {0.35,0.66,0.90}{\draw[thin,gray] (0,0) circle (\r);}
    \foreach \ang in {\th,-\th}{\fill (\ang:\rad) circle (1.3pt);}
    \fill (0,0) circle (1.1pt);
    \node[above] at (0,1.72) {Positive branch coordinate};
    \node[below,align=center] at (0,-1.25) {$\zeta_+=\Xi$\quad(weighted area $s_+$)\\$\chi=e^{i\theta_+}$};
  \end{scope}
  \begin{scope}[shift={(8.1,0)}]
    \draw[semithick] (0,0) circle (1.20);
    \draw[semithick,dashed] (0,-1.20)--(0,1.20);
    \fill (-0.60,0) circle (1.6pt) node[below=3pt] {$-r$};
    \fill (0.60,0) circle (1.6pt) node[below=3pt] {$r$};
    \node[font=\footnotesize] at (-0.56,0.43) {$H>0$};
    \node[font=\footnotesize] at (0.56,0.43) {$H<0$};
    \node[fill=white,inner sep=1pt] at (0,0.0) {$\Gamma$};
    \node[above] at (0,1.72) {Auxiliary disk};
    \node[above right] at (0,1.20) {$i$};
    \node[below right] at (0,-1.20) {$-i$};
    \node[below] at (0,-1.62) {$r=\tfrac12$ (drawing only)};
  \end{scope}
  \begin{scope}[shift={(14.10,0)}]
    \draw[semithick] (-\th:\rad) arc[start angle=-\th,end angle=\th,radius=\rad];
    \draw[semithick,dashed] (\th:\rad) arc[start angle=\th,end angle=360-\th,radius=\rad];
    \foreach \r in {0.35,0.66,0.90}{\draw[thin,gray] (0,0) circle (\r);}
    \foreach \ang in {\th,-\th}{\fill (\ang:\rad) circle (1.3pt);}
    \fill (0,0) circle (1.1pt);
    \node[above] at (0,1.72) {Negative branch coordinate};
    \node[below,align=center] at (0,-1.25) {$\zeta_-=1/\Xi$\quad(weighted area $s_-$)\\$\chi=e^{-i\theta_-}$};
  \end{scope}
  \draw[->] (6.55,0.10)--node[above] {$\Xi|_{\{H>0\}}$}(3.45,0.10);
  \draw[->] (9.65,0.10)--node[above] {$(1/\Xi)|_{\{H<0\}}$}(12.70,0.10);
  \draw[semithick,dashed] (2.1,-2.62)--(2.85,-2.62);
  \node[right] at (2.95,-2.62) {interior zero arc / its image};
  \draw[semithick] (9.25,-2.62)--(10,-2.62);
  \node[right] at (10.1,-2.62) {outer boundary arc / its image};
  \node at (8.1,-3.18) {$s_++s_-=A$; equal weighted areas are not assumed.};
\end{tikzpicture}
\caption{Two-pole coordinates on a geometrically normalized auxiliary disk. Each target boundary is divided into the image of the interior zero arc (dashed) and the image of an outer boundary arc (solid). The gray concentric circles are positive Green levels $h_j=-c^{-1}\log|\zeta_j|=t>0$. The coordinate disks carry generally different weights. Writing $\zeta_j=|\zeta_j|e^{i\theta_j}$, corresponding interface points satisfy $\zeta_-=\overline{\zeta_+}$ and $\theta_-=-\theta_+$ modulo $2\pi$. The arrows are the conformal restrictions to the indicated sign branches, not maps of the entire auxiliary disk and not isometries.}
\label{fig:green-branches}
\end{figure}

Put
\begin{gather}
 a_j(t)=|\{h_j>t\}|,\qquad
 B_j(t)=\int_{h_j=t}|\nabla h_j|^{-1}\,ds_g,
 \qquad T_j=a_j^{-1},\quad G_j(a)=B_j(T_j(a)),\label{eq:branch}\\
 a_j'=-B_j,\qquad
 T_j(a)=\int_a^{s_j}\frac{ds}{G_j(s)}.\label{eq:clock}
\end{gather}
The inverse $T_j$ expresses the Green level in the superlevel-area variable; $a_j'=-B_j$ and $T_j'=-1/G_j$. We call $T_j$ the Green-time coordinate.

\Needspace{6\baselineskip}
Interior regularity follows directly in branch coordinates.  Write
$g=\rho_j(\zeta)|d\zeta|^2$, with $\rho_j>0$ continuous on $\D$, and put
$r=e^{-ct}$.  Then
\begin{equation}
 a_j(t)=\int_{|\zeta|<r}\rho_j(\zeta)\,dA_{\rm E}(\zeta),
 \qquad
 B_j(t)=cr^2\int_0^{2\pi}\rho_j(re^{i\theta})\,d\theta>0.
 \label{eq:LL-coarea-regularity}
\end{equation}
The polar-coordinate fundamental theorem of calculus gives
\[
 B_j\in C(0,\infty),\qquad a_j\in C^1(0,\infty),\qquad
 T_j\in C^1(0,s_j),\qquad G_j\in C(0,s_j),
 \qquad a_j'=-B_j.
\]
Inverse-function regularity uses $a_j'<0$.  No extension of $B_j$ or $G_j$
to the zero-time endpoint is asserted.  The zero set has zero area, so
$a_j(0+)=s_j$, while $a_j(t)\to0$ as $t\to\infty$.  Coarea on compact time intervals, followed by these limits, gives
\begin{equation}
 d(H_*dv_g)(t)=b_H(t)\,dt,\qquad
 b_H(t)=\begin{cases}B_+(t),&t>0,\\ B_-(-t),&t<0,\end{cases}
 \quad \int_0^\infty B_j(t)\,dt=s_j.
 \label{eq:green-time-density}
\end{equation}
There is no residual mass at zero or the poles. The densities may be unbounded near zero but are integrable.

For the geometric comparison, write $\ell_j(t)=\operatorname{Length}_g\{h_j=t\}$. Cauchy--Schwarz, unit flux and Lemma~\ref{lem:local-bol} on each compact superlevel disk yield
\begin{equation}
 B_j(t)\ge \ell_j(t)^2\ge a_j(t)(4\pi-Ka_j(t)),
 \qquad G_j(a)\ge G_K(a).
 \label{eq:order}
\end{equation}
All model coefficients are positive on the relevant intervals. The coarea-coefficient excess has the decomposition
\[
 B_j(t)-G_K(a_j(t))
 =[B_j(t)-\ell_j(t)^2]+[\ell_j(t)^2-G_K(a_j(t))].
\]
These are the nonnegative Cauchy--Schwarz and Bol comparison deficits, respectively.

\subsection{Admissible trial functions and energy isotropy}
Equip the real space $\mathcal H=H^1(\mathbb R;\mathbb R)$ with squared norm
\begin{equation}
 \mathcal N(u)=\int_{\mathbb R}(u'^2+c^2u^2)\,dt,
 \qquad |u(t)|^2\le\mathcal N(u)/(2c).
 \label{eq:norm}
\end{equation}
For smooth compactly supported $u$, the evaluation bound follows from
$\int_{-\infty}^t(u'-cu)^2+\int_t^\infty(u'+cu)^2=\mathcal N(u)-2c|u(t)|^2$.
Density extends it to the continuous representatives of all $H^1$ functions.

\Needspace{6\baselineskip}
\begin{lemma}\label{lem:energy}
For real $u\in\mathcal H$ and a distinct interior Green pair,
\begin{equation}
 F=u(H)\chi\in H^1(\Omega;\mathbb C),\quad
 M(F)=\int_\Omega u(H)^2\,dv_g,\quad E(F)=\mathcal N(u),
 \label{eq:energy}
\end{equation}
and, writing $F_1=\Re F$, $F_2=\Im F$,
\begin{equation}
 Q_F:=\left(\int_\Omega\langle\nabla F_i,\nabla F_j\rangle\,dv_g\right)_{i,j=1}^2
 =\frac{\mathcal N(u)}2I_2.
 \label{eq:green-energy-isotropy}
\end{equation}
For $u\ne0$ the real components are linearly independent in $L^2$.
\end{lemma}
\begin{proof}
\emph{Compactly supported profiles.} For $u\in C_c^\infty(\mathbb R)$, the globally defined $u(H)\chi$ is smooth away from the poles and vanishes in their neighborhoods. It is smooth across $\Gamma$ because $\Xi$ is nonzero there. On the punctured domain use the real one-form $\alpha=-i\overline\chi\,d\chi$. Locally $\alpha=d\arg\Xi$, so Cauchy--Riemann gives $\langle dH,\alpha\rangle_g=0$ and $|\alpha|_g=c|dH|_g$. Consequently
\[
 dF=\chi\bigl(u'(H)dH+i u(H)\alpha\bigr),\qquad
 |\nabla F|^2=(u'(H)^2+c^2u(H)^2)|\nabla H|^2.
\]
Coarea and unit flux give \eqref{eq:energy}. Since $|F|\le\|u\|_\infty$ and the area is finite, the local construction lies in the global form domain of Lemma~\ref{lem:LL-form}. Both branches use the same $\chi$ and its globally defined one-form $\alpha$, whose periods around the poles are $c$ and $-c$.
\emph{Angular energy.} On the two conformal branch disks the fields are
$g_+(r)e^{i\theta}$ and $g_-(r)e^{-i\theta}$, with
$g_\pm(r)=u(\pm[-c^{-1}\log r])$. For either radial factor,
\[
 E(g\cos\theta)=E(g\sin\theta)
 =\pi\int_0^1(rg'^2+g^2/r)\,dr,
 \qquad E(g\cos\theta,g\sin\theta)=0.
\]
Thus each branch has a Dirichlet-energy matrix that is a scalar multiple of the identity, even when $g_+\ne g_-$. Summing the two branch contributions preserves this property.

\emph{General profiles.} Choose $u_n\in C_c^\infty(\mathbb R)$ converging to $u$ in $\mathcal H$. By \eqref{eq:norm}, the squared mass norm of a trial-function difference is at most $A\|u_n-u_m\|_\infty^2\le A\mathcal N(u_n-u_m)/(2c)$, and its Dirichlet energy is $\mathcal N(u_n-u_m)$. Hence the trial functions converge in the physical form norm, and the matrix identities pass to the limit. A nonzero real linear combination vanishing in $L^2$ would have zero energy, contradicting $Q_F=(\mathcal N(u)/2)I_2$ for $u\ne0$.
\end{proof}

\subsection{The kernel operator and the minimizing profile}
For a finite positive measure space $(X,\mu)$ and an almost everywhere finite measurable $h:X\to\mathbb R$, set
\begin{equation}
 D_h(u)=\int_Xu(h)^2\,d\mu,\qquad
 \kappa(h)=\inf_{u\in\mathcal H,\ D_h(u)>0}\frac{\mathcal N(u)}{D_h(u)}.
 \label{eq:kappa}
\end{equation}
For a distinct interior pair and the trial family $F=u(H)\chi$, Lemma~\ref{lem:energy}
identifies $\kappa(H)$ with the infimum of the physical quotient
$E(F)/M(F)$ over nonzero admissible profiles. The general measure
formulation also accommodates the boundary and coalescing parameters
needed for the moment selection in Section~\ref{sec:center}.
Write $\rho_h=h_*\mu$ for the time measure. Since $D_h(u)=\int_{\mathbb R}u^2\,d\rho_h$, the quotient depends only on this pushforward measure. The signed Green function corresponds to $h=H$. Under $u_+(t)=u(t)$ and $u_-(t)=u(-t)$ for $t>0$, the admissible profile space $\mathcal H$ is identified with
\begin{equation}
 \{(u_+,u_-)\in H^1(0,\infty)^2:u_+(0)=u_-(0)\}.
 \label{eq:joint-form-domain}
\end{equation}
Indeed, two half-line functions in $H^1(0,\infty)$ define an $H^1(\mathbb R)$ function after reflection precisely when their traces at $0$ agree; a nonzero trace jump would produce a Dirac mass in the distributional first derivative. No derivative condition is imposed on trial profiles. Such conditions arise only from the optimizing equation; independent branch minimizations define a different problem.

\begin{proposition}\label{prop:kernel}
Assume $0<\mu(X)<\infty$. The operator
\begin{equation}
 R_c(t,s)=\frac{e^{-c|t-s|}}{2c},\qquad
 (\mathcal T_h f)(x)=\int_X R_c(h(x),h(y))f(y)\,d\mu(y)
 \label{eq:kernel}
\end{equation}
is compact, self-adjoint and positive semidefinite on $L^2(X,\mu)$. Its strictly positive integral kernel also makes it positivity improving; this does not assert injectivity. Its largest eigenvalue $\Lambda(h)>0$ is simple and has a unique positive eigenfunction $e_h$ normalized by $\|e_h\|_{L^2(\mu)}=1$. Moreover
\begin{equation}
 \kappa(h)=\Lambda(h)^{-1},\qquad
 u_h(t)=\kappa(h)\int_XR_c(t,h(y))e_h(y)\,d\mu(y)
 \label{eq:profile}
\end{equation}
satisfies $u_h\in\mathcal H$, $u_h\circ h=e_h$ in $L^2(\mu)$, $D_h(u_h)=1$, $\mathcal N(u_h)=\kappa(h)$, and $u_h(t)>0$. Every nonzero real minimizer is a scalar multiple of $u_h$.
\end{proposition}
\begin{proof}
The bounded symmetric kernel defines a Hilbert--Schmidt operator.
The jump of $\partial_tR_c(t,s)$ at $t=s$ is $-1$, giving the reproducing identity
$\langle u,R_c(\cdot,s)\rangle_{\mathcal H}=u(s)$.
The bounded map $J_hu=u\circ h$ therefore has adjoint
$J_h^*f=\int_XR_c(\cdot,h(y))f(y)\,d\mu(y)$, a Bochner integral in the separable space $\mathcal H$. Thus $\mathcal T_h=J_hJ_h^*\ge0$.
\emph{The principal eigenfunction.} The constant $1\in L^2(\mu)$ gives $\langle1,\mathcal T_h1\rangle_{L^2(\mu)}>0$, hence $\Lambda>0$. Replacing a sign-changing eigenfunction for the largest eigenvalue by its absolute value strictly increases its Rayleigh quotient, since the kernel is strictly positive almost everywhere. The eigenvalue equation then gives strict positivity. Two independent eigenfunctions for the largest eigenvalue would give one orthogonal to a positive one, hence a sign-changing eigenfunction for that eigenvalue. This proves simplicity.

\emph{The minimizing profile.} The variational definition gives
\[
 \kappa(h)^{-1}
 =\sup_{u\in\mathcal H\setminus\{0\}}\frac{D_h(u)}{\mathcal N(u)}
 =\|J_h\|^2=\|J_hJ_h^*\|=\Lambda(h).
\]
Substituting $u_h=\kappa J_h^*e_h$ yields $J_hu_h=e_h$ and $\mathcal N(u_h)=\kappa$; the integral is positive at every finite $t$. Distributional testing gives
\begin{equation}
 -u_h''+c^2u_h=\kappa(h)u_h\,d\rho_h.
 \label{eq:measureode}
\end{equation}
\emph{Uniqueness.} A real minimizer $v$ satisfies
$J_h^*J_hv=\kappa(h)^{-1}v$. Hence $J_hv$ belongs to the simple principal
eigenspace of $\mathcal T_h$ and is proportional to $e_h$. The identity
$v=\kappa(h)J_h^*(J_hv)$ then makes $v$ proportional to $u_h$.

\end{proof}
Here $R_c$ is the whole-line Green kernel of $-d^2/dt^2+c^2$.
It is also the reproducing kernel of $\mathcal H$ with its energy inner product.

By \eqref{eq:measureode}, $u_h''$ is a locally finite signed measure, so $u_h'$ has a $BV_{\rm loc}$ representative with one-sided traces. At an atom $b_0=\rho_h(\{0\})$,
\begin{equation}
 u_h'(0+)-u_h'(0-)=-\kappa(h)b_0u_h(0).
 \label{eq:atomic-interface}
\end{equation}
In the half-line coordinates of \eqref{eq:joint-form-domain}, $u_+'(0)=u_h'(0+)$ and $u_-'(0)=-u_h'(0-)$: the derivatives point from the interface into their respective branches, not along the exterior normals of the half-lines. Their sum is $-\kappa b_0u_h(0)$. These derivative conditions follow from the optimizing equation; the form domain itself requires only a common function trace.

For distinct interior Green poles and for model branches one has $b_0=0$, so the two inward half-line derivatives sum to zero. At a coalescing-pole parameter, the time measure is $A\delta_0$ and the derivative sum is $-\kappa A u_h(0)$. Atomlessness alone leaves open the possibility of a singular continuous part in $u_h''$.

For $h=H$, \eqref{eq:green-time-density} gives $u_H\in W^{2,1}_{\rm loc}(\mathbb R)$ with locally absolutely continuous first derivative. The flat-disk example in Appendix~\ref{app:regularity} has logarithmically unbounded time density and shows that $C^2$ regularity at zero can fail.

\begin{lemma}\label{lem:positive-ground-state}
In the setting of Proposition~\ref{prop:kernel}, suppose that $u\in\mathcal H$ is strictly positive and satisfies
\[
 -u''+c^2u=\lambda u\,d\rho_h
\]
in distributions, with $\lambda>0$. Then $\lambda=\kappa(h)$ and $u$ is a positive scalar multiple of $u_h$.
\end{lemma}
\begin{proof}
By \eqref{eq:norm}, $e=J_hu=u\circ h$ is bounded and belongs to $L^2(\mu)$. The function $u_e=\lambda J_h^*e\in\mathcal H$ has the same distributional right-hand side as $u$. Hence $w=u-u_e\in\mathcal H$ solves $-w''+c^2w=0$. Testing this equation with $w$, justified by density in $\mathcal H$, gives $\mathcal N(w)=0$. Consequently $u=\lambda J_h^*e$ and $\mathcal T_he=\lambda^{-1}e$. Since $e>0$ and the principal eigenfunction $e_h>0$, their $L^2(\mu)$ inner product is positive. Self-adjointness rules out different eigenvalues, so $\lambda^{-1}=\Lambda(h)$. Simplicity gives $e=\alpha e_h$ with $\alpha>0$, and then $u=\alpha u_h$ by \eqref{eq:profile}.
\end{proof}

\subsection{Coefficient comparison and quantitative deficit}
In area coordinates, the two energy terms vary oppositely with the
coefficient, so coefficient order alone does not give a fixed-profile
energy comparison. We instead compare the positive kernel through signed
Green-time distances.
Let $s_\pm>0$ and let $G_j$ be positive finite measurable functions, identified almost everywhere, such that
$T_j(a)=\int_a^{s_j}G_j(s)^{-1}\,ds<\infty$ for $0<a<s_j$.
All coefficient comparisons are made on the same two component intervals $(0,s_+)$ and $(0,s_-)$, each equipped with Lebesgue measure. On $Y=(0,s_+)\sqcup(0,s_-)$ put
$x_+(a)=T_+(a)$ and $x_-(a)=-T_-(a)$.
The resulting value is denoted $\kappa(G_+,G_-)$. It agrees with the value generated by the signed Green function since coarea gives $D_H(u)=\sum_j\int_0^{s_j}u(x_j(a))^2\,da$.

\begin{proposition}\label{prop:compression}
Suppose $G_j,\widehat G_j$ satisfy the preceding hypotheses and $G_j\ge\widehat G_j$ almost everywhere. With $\widehat\kappa=\kappa(\widehat G_+,\widehat G_-)$ and positive unit reference eigenfunction $\widehat e$ on $Y$, define
\begin{equation}
 d=\iint_{Y^2}[R_c(x(\alpha),x(\gamma))-R_c(\widehat x(\alpha),\widehat x(\gamma))]
 \widehat e(\alpha)\widehat e(\gamma)\,d\alpha\,d\gamma.
 \label{eq:kernelincrement}
\end{equation}
Then
\begin{equation}
 d\ge0,\qquad \kappa(G_+,G_-)\le\frac{\widehat\kappa}{1+\widehat\kappa d}
 \le\widehat\kappa.
 \label{eq:compression}
\end{equation}
Under the displayed coefficient order, equality of coupled values holds exactly when both coefficient pairs agree almost everywhere. If $G_j>\widehat G_j$ on a set of positive measure in either branch, then $d>0$.
\end{proposition}
\begin{proof}
For two arguments in the same branch and in opposite branches, respectively,
\[
 |T_j(a)-T_j(b)|=\left|\int_a^b G_j(s)^{-1}\,ds\right|,
 \qquad |x_+(a)-x_-(b)|=T_+(a)+T_-(b).
\]
Under the coefficient order, none of these Green-time distances increases, so the kernel is pointwise no smaller. Merely ordering the Green-time coordinates pointwise would not control same-branch distances. Testing on $\widehat e$ gives $\Lambda\ge1/\widehat\kappa+d$, proving the estimate.

If $G_j>\widehat G_j$ on a set of positive measure, choose an interior interval $(a_0,a_1)$ with
$d_0=\int_{a_0}^{a_1}(\widehat G_j^{-1}-G_j^{-1})>0$.
For every $a<a_0$, the Green-time coordinate $T_j(a)$ decreases by at least $d_0$. Hence the cross-branch distance from such an $a$ to every point of the other branch decreases by at least $d_0$. Since $(0,a_0)$ and the other branch both have positive Lebesgue measure, the kernel increment is strictly positive on a set of positive product measure. Pairing it with the strictly positive reference eigenfunction $\widehat e$ gives $d>0$. Same-branch pairs below $a_0$ may retain their distance; cross-branch pairs suffice for strictness. Conversely, almost-everywhere equal coefficients give identical Green-time coordinates.
\end{proof}

The comparison uses the pointwise order of the positive kernel under contraction of Green-time distances. Such an order does not imply Loewner order: shortening a two-point separation produces the kernel difference $\left(\begin{smallmatrix}0&d\\d&0\end{smallmatrix}\right)$, with eigenvalues $\pm d$.

Figure~\ref{fig:time-contraction} distinguishes the two types of separation controlled by the coefficient order.

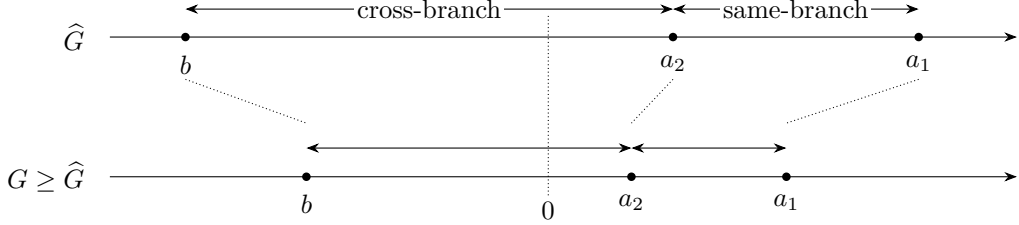
\begin{figure}[htbp]
\centering
\begin{tikzpicture}[x=1cm,y=1cm,font=\small,>=Stealth]
  \node[anchor=west] at (-7.7,3.62) {\textbf{(a)} Signed area-to-level coordinates};
  \draw[<->] (-6.4,2.55)--(6.4,2.55);
  \draw (0,2.46)--(0,2.64) node[above=2pt] {$0$};
  \node[left] at (-6.4,2.55) {$-\infty$};
  \node[right] at (6.4,2.55) {$+\infty$};
  \node[above] at (-3.45,2.58) {$x_-(a)=-T_-(a)$};
  \node[above] at (3.45,2.58) {$x_+(a)=T_+(a)$};
  \draw[->] (-5.2,1.99)--node[below] {$a\uparrow s_-$}(-1.45,1.99);
  \draw[->] (5.2,1.99)--node[below] {$a\uparrow s_+$}(1.45,1.99);
  \node[fill=white,inner sep=2pt] at (0,1.80) {$u_-(0)=u_+(0)$};
  \node[anchor=west] at (-7.7,0.95) {\textbf{(b)} Reference and contracted separations};
  \draw[->] (-5.8,0.2)--(6.2,0.2);
  \draw[->] (-5.8,-1.65)--(6.2,-1.65);
  \node[anchor=east] at (-6.0,0.2) {$\widehat G$};
  \node[anchor=east] at (-6.0,-1.65) {$G\ge\widehat G$};
  \draw[densely dotted] (0,0.48)--(0,-1.93);
  \node[below] at (0,-1.85) {$0$};
  \foreach \x/\lab in {-4.8/{$b$},1.65/{$a_2$},4.9/{$a_1$}}{
    \fill (\x,0.2) circle (1.6pt);
    \node[below=3pt] at (\x,0.2) {\lab};
  }
  \foreach \x/\lab in {-3.2/{$b$},1.10/{$a_2$},3.15/{$a_1$}}{
    \fill (\x,-1.65) circle (1.6pt);
    \node[below=3pt] at (\x,-1.65) {\lab};
  }
  \draw[<->] (-4.8,0.57)--node[fill=white,inner sep=1pt] {cross-branch}(1.65,0.57);
  \draw[<->] (1.65,0.57)--node[fill=white,inner sep=1pt] {same-branch}(4.9,0.57);
  \draw[<->] (-3.2,-1.27)--(1.10,-1.27);
  \draw[<->] (1.10,-1.27)--(3.15,-1.27);
  \draw[densely dotted] (-4.8,-0.36)--(-3.2,-0.94);
  \draw[densely dotted] (1.65,-0.36)--(1.10,-0.94);
  \draw[densely dotted] (4.9,-0.36)--(3.15,-0.94);
\end{tikzpicture}
\[
 |x_+(a_1)-x_+(a_2)|=\int_{a_1}^{a_2}\frac{ds}{G_+(s)},
 \qquad |x_+(a_2)-x_-(b)|=T_+(a_2)+T_-(b).
\]
\caption{Signed Green-time coordinates and the two types of separation. In panel (b), $b\in(0,s_-)$ and $0<a_1<a_2<s_+$. Matching labels identify the same area point on the reference and comparison axes. Under $G_j\ge\widehat G_j$, neither marked separation increases; the contraction need not be uniform. The displayed formulas correspond to the same-branch and cross-branch segments, respectively, and hold with hats for the reference axes. Infinite tails in panel (a) refer to geometric and model branches. The comparison is pointwise kernel order, not Loewner order, and does not supply the phase information needed for centering.}
\label{fig:time-contraction}
\end{figure}

\FloatBarrier
\section{The model maximum and strict domain comparison}\label{sec:model}
For $s_1+s_2=A$, $s_j>0$, let $\kappa_K(s_1,s_2)$ be the coupled value with coefficient $G_K$ on each branch and Green-time coordinate \eqref{eq:modelclock}. It is the minimum of the whole-line Rayleigh quotient with the two branches sharing their interface trace, not the ordinary lowest Neumann value of the disconnected union, which has zero modes.
Label the model branches by $s_+=s_1$, $s_-=s_2$, and write $u_\pm(t)=u(\pm t)=f_\pm(a_{K,s_\pm}(t))$ for $t>0$. Since $(a_{K,s_j})'=-G_K(a_{K,s_j})$, change of variables gives
\begin{align}
 \mathcal N(u)&=\sum_{j=\pm}\int_0^{s_j}
       \left(G_K|f_j'|^2+\frac{c^2}{G_K}|f_j|^2\right)\,da,
       \label{eq:joint-area-energy}\\
 D_{\rm mod}(u)&=\sum_{j=\pm}\int_0^{s_j}|f_j|^2\,da,
 \qquad f_+(s_+)=f_-(s_-).\label{eq:joint-area-mass}
\end{align}
The form domain in area coordinates consists of locally absolutely continuous functions $f_j$ with finite displayed energy and a common interface trace at $a=s_j$. These conditions are equivalent to \eqref{eq:joint-form-domain}: the energy identity yields the two $H^1$ half-line profiles, whose traces glue. Identities follow first on compact subintervals and then by exhaustion. No branch Neumann condition is imposed. The $O(\sqrt a)$ center expansion below belongs to regular eigenfunctions, not to every trial element.

For a fixed spectral parameter $\lambda>0$ let $v_\lambda$ be the regular solution
\begin{equation}
 -(G_Kv_\lambda')'+\frac{c^2}{G_K}v_\lambda=\lambda v_\lambda,
 \qquad v_\lambda(a)\sim\sqrt a.
 \label{eq:modelODE}
\end{equation}
\Needspace{4\baselineskip}
Let $a_D(\lambda)$ be its first zero in the model area interval, or its endpoint if no zero occurs. That endpoint is $4\pi/K$ for $K>0$ and $+\infty$ otherwise. On $I_\lambda=(0,a_D(\lambda))$ put
\begin{equation}
 R_\lambda(a)=G_K(a)\frac{v_\lambda'(a)}{v_\lambda(a)}.
 \label{eq:response}
\end{equation}
On $I_\lambda$, $v_\lambda>0$ but $R_\lambda$ may have either sign. This weighted logarithmic derivative is the boundary-flux response: at $a=s$, it is the outward conormal flux $G_K(s)v_\lambda'(s)$ in the area coordinate divided by $v_\lambda(s)$.

If $s_\pm\in I_\lambda$, normalize each branch to interface value $1$
by $f_j=v_\lambda/v_\lambda(s_j)$ and join them to a profile $U$.
Integration of the model equation gives
\[
 \mathcal N(U)-\lambda D_{\rm mod}(U)
 =R_\lambda(s_+)+R_\lambda(s_-).
\]
For $K\ge0$, concavity of $R_\beta$ will compare area splits by summing the
two branch responses. If a branch reaches the first Dirichlet zero,
a zero-extension argument replaces this normalization.

\begin{lemma}\label{lem:concavity}
For $K\ge0$ and $\lambda>2K$, $R_\lambda''<0$ on all of $I_\lambda$.
\end{lemma}
\begin{proof}
Differentiating the equation and its regular center series gives
\begin{align}
 R'&=(c^2-R^2)/G_K-\lambda,\label{eq:riccati}\\
 R(a)&=c-\frac\lambda2a-\frac{\lambda(\lambda-2K)}{24c}a^2+O(a^3),\label{eq:series}\\
 G_KR'''&=-2(R+G_K')R''-2[(R')^2-KR'-K\lambda].\label{eq:third}
\end{align}
The last identity follows by twice differentiating $G_KR'=c^2-R^2-\lambda G_K$ and using $G_K''=-2K$. The center expansion gives $R''(0+)=-\lambda(\lambda-2K)/(12c)<0$. If $a_*$ were the first zero of $R_\lambda''$ in $I_\lambda$, then $R_\lambda'$ would be decreasing on $(0,a_*)$, where $R_\lambda''<0$. Hence $R_\lambda'(a_*)\le R_\lambda'(0+)=-\lambda/2$, and
\[
 (R')^2-KR'-K\lambda\ge\lambda(\lambda-2K)/4>0.
\]
Equation~\eqref{eq:third} would give $R_\lambda'''(a_*)<0$, contradicting first contact of $R_\lambda''$ with zero from below.
\end{proof}

For $K\ge0$, Lemma~\ref{lem:angularmodel} gives $\beta>2K$ and
\begin{equation}
 v_\beta>0\text{ on }(0,m],\qquad R_\beta(m)=0.
 \label{eq:responsezero}
\end{equation}

\begin{proposition}\label{prop:modelsplit}
Fix $K\in\mathbb R$ and $A>0$, with $KA<4\pi$ if $K>0$. For every $s_1,s_2>0$ with $s_1+s_2=A$,
\begin{equation}
 \kappa_K(s_1,s_2)\le\beta_K(A/2),
 \qquad\text{with equality exactly when }s_1=s_2=m.
 \label{eq:modelsplit}
\end{equation}
\end{proposition}
\begin{proof}
\emph{The equal-area model, for every admissible $K$.} In the equal-area model $s_1=s_2=m$ the entire time density is $B_0(|t|)$. The positive even extension $g_0(|t|)$ solves the weak equation with spectral parameter $\beta$: $g_0'(0)=0$ prevents a Dirac mass at $t=0$ in the distributional second derivative. Lemma~\ref{lem:positive-ground-state} shows that it is a positive scalar multiple of the normalized minimizing profile and gives $\kappa_K(m,m)=\beta$. Its trial mass is $2M_0$, so the unit-mass profile is $g_0(|t|)/\sqrt{2M_0}$.

\emph{Nonnegative curvature bound: both areas lie in $I_\beta$.} For $K\ge0$, suppose first $s_1,s_2\in I_\beta$. The branch functions
$f_j(a)=v_\beta(a)/v_\beta(s_j)$ have common interface trace $1$ and give an admissible whole-line profile. Multiplying the ODE by $f_j$ and integrating yields
\[
 \int_0^{s_j}(G_Kf_j'^2+c^2f_j^2/G_K-\beta f_j^2)\,da=R_\beta(s_j).
\]
The center term vanishes. Strict concavity gives
$R_\beta(s_1)+R_\beta(s_2)\le2R_\beta(m)=0$, strictly for unequal areas. Thus the combined quotient is at most $\beta$.

\emph{Nonnegative curvature bound: a Dirichlet zero is reached.} For the remaining $K\ge0$ case, the two branch areas do not both belong to $I_\beta$. Put
\[
 s_D=a_D(\beta),\qquad
 s_L=\max\{s_1,s_2\},\qquad s_S=\min\{s_1,s_2\}.
\]
Since $v_\beta>0$ on $(0,m]$ and at least one branch area is outside $I_\beta$, one has
\[
 m<s_D\le s_L<A.
\]
The right endpoint of the model area interval is larger than $A$: it is $4\pi/K$ when $K>0$ and $+\infty$ when $K=0$. Thus $s_D$ is an actual interior Dirichlet zero, not the endpoint assigned in the absence of a zero, and $G_K(s_D)>0$. Since $s_D>m$, only the longer branch can reach it. Define the long- and short-branch functions by
\[
 f_L(a)=
 \begin{cases}
  v_\beta(a),&0<a<s_D,\\
  0,&s_D\le a<s_L,
 \end{cases}
 \qquad
 f_S(a)=0\quad(0<a<s_S).
\]
The regular ODE solution has
\[
 v_\beta(a)=\sqrt a+O(a^{3/2}),\qquad
 v_\beta'(a)=\frac1{2\sqrt a}+O(\sqrt a).
\]
Since $G_K(a)\sim2ca$, both $G_K|v_\beta'|^2$ and $c^2|v_\beta|^2/G_K$ are integrable at zero. At the regular endpoint $s_D$, $v_\beta(s_D)=0$, so the zero extension has no function jump and finite energy, and
\[
 f_L(s_L)=f_S(s_S)=0.
\]
Thus the two branches have the required common interface trace. By the equivalence of this area-coordinate form domain with \eqref{eq:joint-form-domain}, the Green-time change gives a nonzero admissible profile in $H^1(\mathbb R)$. Multiplying the model ODE by $v_\beta$ and integrating over $(0,s_D)$ gives
\[
 \int_0^{s_D}\left(G_K|v_\beta'|^2+\frac{c^2}{G_K}|v_\beta|^2\right)\,da
 =\beta\int_0^{s_D}|v_\beta|^2\,da,
\]
so this profile has Rayleigh quotient exactly $\beta$. If it attained the minimum, Proposition~\ref{prop:kernel} would make it a scalar multiple of the strictly positive minimizing profile. This is impossible because $f_S$ vanishes identically on the nonempty shorter branch. Hence $\kappa_K(s_1,s_2)<\beta$ in this case.

\emph{Negative curvature bound.} For $K<0$, the explicit superlevel area in \eqref{eq:model-arbitrary-superlevel} satisfies, for every $t>0$,
\begin{equation}
 \partial_s^2 a_{K,s}(t)
 =\frac{2(4\pi)^2K\vartheta(1-\vartheta)}{[4\pi-Ks(1-\vartheta)]^3}<0.
 \label{eq:negative-model-concavity}
\end{equation}
Hence $a_{K,s_1}+a_{K,s_2}\le2a_0$, strictly for unequal areas. For the even test $u(t)=g_0(|t|)$, integration by parts on each branch gives
\[
 D_j=\int_0^\infty g_0^2(-a_{K,s_j}')\,dt
 =s_j-\int_0^\infty a_{K,s_j}w_0\,dt.
\]
Since $w_0>0$ and $\int_0^\infty w_0=1$, the smaller cumulative
area increases the trial mass. Consequently,
\begin{gather}
 \mathcal N(u)=2\beta M_0,\quad
 D_{\rm mod}(u)=2M_0+\mathcal L_K,\quad
 \mathcal L_K=\int_0^\infty(2a_0-a_{K,s_1}-a_{K,s_2})w_0\,dt,\label{eq:negative-model-mass}\\
 \kappa_K(s_1,s_2)\le\frac{2\beta M_0}{2M_0+\mathcal L_K}<\beta
 \quad(s_1\ne s_2).\label{eq:negative-model-test}
\end{gather}
\end{proof}
At $K=0$ the superlevel area is linear in $s$, so strictness there comes from the fixed-spectral-parameter proof, not from strict Jensen concavity.

\begin{proposition}\label{prop:critical-model}
Let $K>0$ and $A=4\pi/K$. For every $s_1,s_2>0$ with
$s_1+s_2=A$,
\begin{equation}
 \kappa_K(s_1,s_2)=2K=\beta_K(A/2).
 \label{eq:critical-model}
\end{equation}
\end{proposition}
\begin{proof}
The half-area disk is a hemisphere. Its positive first-angular
Neumann factor is $S_K$, with eigenvalue $2K$, so
$\beta_K(A/2)=2K$ by Lemma~\ref{lem:angularmodel}.
On $0<a<A$, the positive regular solution
$v(a)=\sqrt{G_K(a)}$ satisfies
\[
 -(G_Kv')'+\frac{c^2}{G_K}v=2Kv,
 \qquad R_{2K}(a)=\frac{G_K'(a)}2=c-Ka.
\]
Indeed $G_K''=-2K$ and $G_K'^2+4KG_K=4c^2$.
Use $f_j(a)=v(a)/v(s_j)$ in the two branch intervals and let
$U(\pm t)=f_\pm(a_{K,s_\pm}(t))$. This is a positive
$H^1(\mathbb R)$ profile with common interface value $1$;
the center expansion gives exponential decay at both infinite ends.
Its interface derivatives satisfy
\[
 U'(0+)=-R_{2K}(s_1),\qquad U'(0-)=R_{2K}(s_2),
 \qquad R_{2K}(s_1)+R_{2K}(s_2)=2c-KA=0.
\]
Thus the branch equations glue without an interface atom to the
whole-line weak equation with spectral parameter $2K$.
Lemma~\ref{lem:positive-ground-state} identifies this positive
profile with the minimizing profile and proves the assertion.
\end{proof}
At the critical area every positive model split has the same value;
the equal-area rigidity of Proposition~\ref{prop:modelsplit} does
not extend to this endpoint.

\begin{theorem}\label{thm:transmission}
Under the strict area assumptions of Theorem~\ref{thm:main}, a distinct
interior pair of Green poles cannot satisfy both identities
\begin{equation}
 G_+=G_K\ \hbox{on }(0,s_+),\qquad
 G_-=G_K\ \hbox{on }(0,s_-).
 \label{eq:doubleequality}
\end{equation}
Consequently at least one coefficient is strictly larger on an interior
interval, and
\begin{equation}
 \kappa(H)<\kappa_K(s_+,s_-)\le\beta.
 \label{eq:positiveallpair}
\end{equation}
Almost-everywhere equality of both coefficients is equally impossible.
\end{theorem}
\begin{proof}
Near any interior point of the zero interface, the assumed equality will
determine the density smoothly up to each side. The contradiction will
come from the resulting interface term in its distributional curvature.
Assume \eqref{eq:doubleequality}. Interior continuity first upgrades
almost-everywhere equality to pointwise equality.  On a positive regular
level of branch $j$, write $q=|\nabla h_j|_g$ and $L=\ell_j(t)$.  Unit flux
and \eqref{eq:order} give
\[
 \int q\,ds_g=1,\qquad
 \int q^{-1}\,ds_g=B_j(t)=L^2=G_K(a_j(t)).
\]
Equality in Cauchy--Schwarz makes $q$ constant along the level; since $q$
is continuous there, this holds pointwise.  Hence
\begin{equation}
 |\nabla h_j|_g^2=G_K(a_j(t))^{-1}\quad\hbox{on }h_j=t,
 \qquad a_j'(t)=-G_K(a_j(t)).
 \label{eq:transnormal}
\end{equation}
The model coefficient is positive at $s_j$.  Together with
$a_j(0+)=s_j$, its smooth autonomous equation extends $a_j$ smoothly to
$t=0$ from the right.  This endpoint regularity is a consequence of the
assumed equality, not an assertion for general Green-time coefficients.

Fix an interior point of the zero curve $\Gamma$ and a small conformal
coordinate neighborhood $U$ where $H$ is Euclidean harmonic and
$\nabla_{\rm E}H\ne0$.  Write $g=\sigma|dz|^2$ there.  The density $\sigma$
is positive and continuous and satisfies
$\Delta\log\sigma+2K\sigma\ge0$ in distributions.  Equation
\eqref{eq:transnormal} gives the piecewise identity
\begin{equation}
 \sigma=|\nabla_{\rm E}H|^2
 \begin{cases}
  G_K(a_+(H)),&H>0,\\
  G_K(a_-(-H)),&H<0.
 \end{cases}
 \label{eq:LL-transmission-regularity}
\end{equation}
Thus $\log\sigma$ is smooth up to each side of $\Gamma$ and continuous
across it.  Since the nonvanishing holomorphic derivative of $H$ makes
$\log|\nabla_{\rm E}H|^2$ harmonic, and
\[
 \frac{d^2}{dt^2}\log G_K(a_j(t))
   =G_K''(a_j(t))G_K(a_j(t))=-2K G_K(a_j(t)),
\]
one has $\Delta\log\sigma=-2K\sigma$ away from the interface.
Let $\nu=\nabla_{\rm E}H/|\nabla_{\rm E}H|$ point toward $H>0$.  The
one-sided normal derivatives satisfy
\[
 \partial_\nu(\log\sigma)_+
  -\partial_\nu(\log\sigma)_-
 =-[G_K'(s_+)+G_K'(s_-)]|\nabla_{\rm E}H|
 =-(8\pi-2KA)|\nabla_{\rm E}H|.
\]
Integration by parts on the two sides, using the matching function traces,
therefore gives on $U$ the distributional identity
\begin{equation}
 \Delta\log\sigma+2K\sigma\,dx
   =-(8\pi-2KA)|\nabla_{\rm E}H|\,ds_{{\rm E},\Gamma}.
 \label{eq:criticalarea}
\end{equation}
Its right-hand side is a strictly negative nonzero measure because
$KA<4\pi$ and $\nabla_{\rm E}H\ne0$.  This contradicts the distributional
curvature upper bound.  A nonnegative continuous coefficient deficit must
therefore be positive on an interior interval.  Proposition~\ref{prop:compression}
and Proposition~\ref{prop:modelsplit} finish the proof.
\end{proof}
At $KA=4\pi$, the interface term in \eqref{eq:criticalarea} vanishes.
This explains why the strictness argument stops at the critical area.

\FloatBarrier
\section{Simultaneous centering and the spectral conclusion}\label{sec:center}
To impose the two weighted moment conditions, we return from the pushforward measure $H_*\mu$ to the spatial measure and the phase of the trial function.

\subsection{Parameter extension and continuity}
Fix an absolutely continuous measure $\mu$ on $\D$ with $0<\mu(\D)=A<\infty$; no boundary bounds or regularity of its density are assumed. Extend $B_a(z)$ to $|a|=1$ by $B_a(z)=-a$. On the closed bidisk $\mathcal P=\overline\D^2$, use the same expressions for $H_\xi,\chi_\xi$, $\xi=(a,b)$, with $H_{a,a}=0$ and $\chi_{a,a}=1$. Pole points have zero $\mu$-measure. In this parameter section we use $F_\xi$ and $h_a$ for their pullbacks to $\D$, and use the same notation for the corresponding functions composed with $\Phi^{-1}$ on $\Omega$ when taking physical mass and energy. Define
\begin{equation}
 \eta_\xi=u_{H_\xi}(0)>0,\qquad
 F_\xi=\eta_\xi^{-1}e_{H_\xi}\chi_\xi,\qquad
 M(F_\xi)=\eta_\xi^{-2}.
 \label{eq:Fadapt}
\end{equation}
This normalization fixes the coalescing value $F_{a,a}=1$ and gives
the explicit boundary-face identities used below.
For distinct interior poles this is $v_\xi(H_\xi)\chi_\xi$, where $v_\xi=u_{H_\xi}/\eta_\xi$ and $v_\xi(0)=1$.
The profile $u_h$ has unit weighted mass; $v_\xi$ instead has unit interface value.

\begin{lemma}\label{lem:continuity}
The functions $\kappa_\xi,\eta_\xi$ are continuous on $\mathcal P$; $e_{H_\xi}$ is continuous in $L^2(\mu)$. The family $F_\xi$ is uniformly bounded, continuous in $L^p(\mu)$ for $1\le p<\infty$, and continuous in moments against every fixed $L^1(\mu)$ weight. Bounds are uniform over $\mathcal P$ for this fixed $\mu$, not over all measures of mass $A$.
\end{lemma}
\begin{proof}
Along a convergent parameter sequence the Blaschke quotients, times and phases converge almost everywhere, except possibly at two limiting interior poles. The bounded kernels converge almost everywhere on the product space, hence in Hilbert--Schmidt norm. Therefore they also converge in operator norm, and the variational characterization of the largest eigenvalue gives $\Lambda_n\to\Lambda>0$. Write $\mathcal T_n=\mathcal T_{H_{\xi_n}}$ and $\mathcal T=\mathcal T_{H_\xi}$. The eigenvalue equation gives
\[
 e_n-\Lambda_n^{-1}\mathcal T e_n
 =\Lambda_n^{-1}(\mathcal T_n-\mathcal T)e_n\longrightarrow0
 \quad\text{in }L^2(\mu).
\]
Compactness of $\mathcal T$ gives precompactness of the unit $e_n$. Every subsequential limit is its unique nonnegative unit eigenfunction for the largest eigenvalue, so the whole sequence converges.

Write $J_n=J_{H_{\xi_n}}$ and $J=J_{H_\xi}$. The reproducing identity gives
\begin{equation}
 \|J_n^*-J^*\|_{\mathrm{HS}(L^2(\mu),\mathcal H)}^2
 =\frac1c\int_\D\bigl(1-e^{-c|H_{\xi_n}-H_\xi|}\bigr)\,d\mu\longrightarrow0.
 \label{eq:adjoint-profile-continuity}
\end{equation}
Indeed the squared $\mathcal H$-norm of $R_c(\cdot,t)-R_c(\cdot,s)$ is $(1-e^{-c|t-s|})/c$; integration gives the Hilbert--Schmidt identity, and dominated convergence applies. Hence
$u_{H_{\xi_n}}=\kappa_{\xi_n}J_n^*e_n\to u_{H_\xi}$ strongly in $\mathcal H$. The evaluation bound gives uniform profile convergence and $\eta_{\xi_n}\to\eta_\xi>0$. On compact $\mathcal P$, $\inf\eta_\xi>0$, $\sup\kappa_\xi<\infty$ and $\|u_{H_\xi}\|_{\mathcal H}^2=\kappa_\xi$ give a uniform bound on the trial functions. Choose the representative $F_\xi(z)=\eta_\xi^{-1}u_{H_\xi}(H_\xi(z))\chi_\xi(z)$ off the poles. Uniform profile convergence and almost-everywhere convergence of the times and phases give pointwise convergence. Dominated convergence proves the finite-$p$ and fixed-weight claims.
\end{proof}
Strong $\mathcal H$ convergence of the profiles does not imply $H^1(\Omega)$ convergence after composition with varying $H_\xi$; see the coalescing-pole example below.

The moment limits hold for each fixed weight, rather than uniformly over an $L^1$ unit ball. Varying weights are allowed when $w_n\to w$ strongly in $L^1(\mu)$: their additional contribution is bounded by $C_\mu\|w_n-w\|_1$, and the fixed-$w$ term tends to zero.

For any time map $h$, reflection gives $\mathcal T_{-h}=\mathcal T_h$, $e_{-h}=e_h$ and $u_{-h}(t)=u_h(-t)$. The unique positive minimizing profile is even if and only if $\rho_h$ is reflection invariant: uniqueness proves sufficiency; reflecting \eqref{eq:measureode} and dividing by the positive even profile proves necessity. Thus exchanging poles conjugates $F$, but does not make each profile even.

For $a\in\D$, write $\psi_a(z)=-c^{-1}\log|B_a(z)|$ in the auxiliary coordinate, and define the one-pole trial function by
\begin{equation}
 h_a(z)=\frac{u_{\psi_a}(\psi_a(z))}{u_{\psi_a}(0)}
         \frac{B_a(z)}{|B_a(z)|},\qquad z\ne a.
 \label{eq:one-pole-normalization}
\end{equation}
The profile is normalized by its value at zero. If $a\in\D$ and $|b|=1$, then $B_b=-b$ and $H_{a,b}=\psi_a$: a boundary parameter leaves an interior Green pole. The reflection identities give
\begin{align}
 F_{a,b}&=-\bar b\,h_a&&(|b|=1),&
 F_{a,b}&=-a\,\overline{h_b}&&(|a|=1),\notag\\
 F_{a,a}&=1,& h_a&=-a&&(|a|=1).
 \label{eq:faces}
\end{align}
At a coalescing-pole parameter $H=0$, the time measure is $A\delta_0$ and the kernel is constant, $\kappa=2c/A$, and $e=\eta=A^{-1/2}$; hence the normalization extends continuously to this parameter.

\subsection{Simultaneous vanishing of the moments}
\begin{proposition}\label{prop:centering}
For each fixed real $w\in L^1(\mu)$ with $\int w\,d\mu=0$, either an interior one-pole field satisfies $\int h_a=\int wh_a=0$, or a distinct interior pair satisfies
\begin{equation}
 \int F_{a,b}\,d\mu=\int wF_{a,b}\,d\mu=0.
 \label{eq:centering}
\end{equation}
\end{proposition}
\begin{proof}
Put $N(a)=(\int h_a,\int wh_a)$ and $\mathcal M(a,b)=(\int F_{a,b},\int wF_{a,b})$.
Continuity and the reality of $w$ give
\begin{align}
 N(a)&=(-Aa,0)&&(|a|=1),\notag\\
 \mathcal M(a,b)&=-\bar bN(a)&&(|b|=1),\notag\\
 \mathcal M(a,b)&=-a\overline{N(b)}&&(|a|=1),\qquad
 \mathcal M(a,a)=(A,0).
 \label{eq:momentfaces}
\end{align}
If $N$ has a zero it is interior, giving the first alternative. Otherwise
we deform the normalized boundary moment map to an explicit model of
degree $-2$. This nonzero degree will force a zero of $\mathcal M$;
the face and diagonal identities will place it at a distinct interior
pair. Set
\[
 f_0=N/|N|,\qquad f_1=N_*/|N_*|,\qquad
 N_*(a)=A(-a,1-|a|^2).
\]
The maps $f_0,f_1:\overline{\mathbb D}\to S^3$ agree on $\partial\mathbb D$, where both equal $(-a,0)$. Gluing $f_0$ to the reverse of $f_1$ along this common boundary gives a map $S^2\to S^3$ representing the obstruction to a homotopy relative to $\partial\mathbb D$. Since $\pi_2(S^3)=0$ \cite[Corollary 4.9]{Hatcher}, the obstruction vanishes and there is a homotopy $f_s$ from $f_0$ to $f_1$ fixed on the boundary. Now set
\[
 L_s(a)=(1-s)|N(a)|+s|N_*(a)|>0.
\]
Then $N_s:=L_s f_s$ is a homotopy through nonzero $\mathbb C^2$-valued maps from $N$ to $N_*$, fixed on the boundary because both boundary lengths equal $A$. Extend it to the two boundary faces by
\begin{equation}
 \mathcal M_s^{(b)}(a,b)=-\bar b\,N_s(a)\quad (|b|=1),
 \qquad
 \mathcal M_s^{(a)}(a,b)=-a\,\overline{N_s(b)}\quad (|a|=1).
 \label{eq:face-homotopy}
\end{equation}
Each face map is nonvanishing. On the corner torus, where $N_s(a)=(-Aa,0)$ and $N_s(b)=(-Ab,0)$ for every $s$, both formulas in \eqref{eq:face-homotopy} equal $(Aa\bar b,0)$. Hence the two face homotopies glue continuously and remain nonvanishing throughout. At $s=0$ they give $\mathcal M|_{\partial\mathcal P}$, while at $s=1$ they give the model boundary map below.

Suppressing the factor $A$, the final boundary map is
\begin{equation}
 \begin{aligned}
 Z_b(a,b)&=\bigl(a\bar b,-\bar b(1-|a|^2)\bigr),&& |b|=1,\\
 Z_a(a,b)&=\bigl(a\bar b,-a(1-|b|^2)\bigr),&& |a|=1.
 \end{aligned}
 \label{eq:degree-faces}
\end{equation}
These maps agree on the corner torus and never vanish. We compute the degree of their normalization $\widehat Z:\partial\mathcal P\to S^3$ directly. Orient $\mathcal P$ by $(\Re a,\Im a,\Re b,\Im b)$ and each boundary face by the outward-first convention. The target $(0,1)\in\C^2$ has exactly two preimages: $(a,b)=(0,-1)$ on the first face and $(-1,0)$ on the second. Neither lies on the corner torus. A positive target chart at $(0,1)$ is $(\Re Z_1,\Im Z_1,\Im Z_2)$.

At $(0,-1)$ use $a=x+iy$, $b=-e^{it}$ with source coordinates $(x,y,t)$. At $(-1,0)$ use $a=-e^{it}$, $b=x+iy$ with coordinates $(t,x,y)$. Both are positive boundary charts: adjoining the outward normal first gives, respectively, $(-e_3,e_1,e_2,-e_4)$ and $(-e_1,-e_2,e_3,e_4)$ in the standard ambient basis. In the stated target chart the two derivatives are
\begin{equation}
 D\widehat Z\big|_{(0,-1)}=
 \begin{pmatrix}-1&0&0\\0&-1&0\\0&0&-1\end{pmatrix},
 \qquad
 D\widehat Z\big|_{(-1,0)}=
 \begin{pmatrix}0&-1&0\\0&0&1\\1&0&0\end{pmatrix}.
 \label{eq:degreejac}
\end{equation}
Normalization does not change these tangent derivatives, since the derivative of the length vanishes at both points. Each determinant is $-1$, so the boundary degree is $-2$. The two regular preimages lie away from the corner torus, so these local degrees compute the degree of the continuous boundary map. The closed bidisk $\mathcal P=\overline\D\times\overline\D$ is a compact convex body in $\mathbb R^4$, hence is homeomorphic to the closed four-ball and has boundary homeomorphic to $S^3$. If $\mathcal M$ were nonvanishing on all of $\mathcal P$, its normalization would extend the normalized boundary map over this four-ball, forcing boundary degree zero. Thus $\mathcal M$ has a zero; \eqref{eq:momentfaces} excludes the boundary and the diagonal, giving the second alternative.
\end{proof}
The selection may depend on the fixed weight $w$. The degree proves existence, not a count of zeros, a common parameter for all weights, or a continuous choice as the domain varies.

\Needspace{6\baselineskip}
\subsection{The one-pole alternative and the spectral conclusion}
We now return to settings \textup{(L)} and \textup{(G)}, with
$\mu=\Phi^*dv_g$. The energy identities below do not require the
strict area condition; the relevant curvature comparison supplies
the upper bound on $\kappa_a$ when the lemma is applied.
\begin{lemma}\label{lem:onepole-boundary}
For interior $a$, $h_a\in H^1(\Omega;\mathbb C)$ is nonzero. With $M_a=\int|h_a|^2$ and normalization $v(0)=1$,
\begin{equation}
 E(h_a)=\kappa_a M_a-c>0,\qquad
 Q_{h_a}=\frac{E(h_a)}2I_2,\qquad E(h_a)/M_a<\kappa_a.
 \label{eq:boundaryquot}
\end{equation}
Here $\kappa_a$ is the whole-line kernel value at the boundary parameter.
In particular, if $\kappa_a\le\beta$, then $E(h_a)/M_a<\beta$.
\end{lemma}
\begin{proof}
The one-pole Green-time variable is nonnegative, so its pushforward measure vanishes on the negative half-line. There \eqref{eq:measureode}, finite energy and $v(0)=1$ give
\begin{equation}
 v(t)=e^{ct},\qquad
 \int_{-\infty}^0(v'^2+c^2v^2)\,dt=c.
 \label{eq:tail}
\end{equation}

The physical trial function occupies the other half-line. The calculation of Lemma~\ref{lem:energy}, restricted to $t\ge0$, gives
\[
 E(h_a)=\int_0^\infty(v'^2+c^2v^2)\,dt
       =\mathcal N(v)-c=\kappa_aM_a-c>0.
\]
The same angular integration makes the Dirichlet-energy matrix of its real and imaginary parts scalar. Smooth approximation of the whole-line profile, as in that lemma, proves $H^1$ admissibility.

The measure has no atom at zero, so \eqref{eq:atomic-interface} and \eqref{eq:tail} give $v'(0+)=c$. In the coordinate $\zeta=B_a$, the radial factor $f(r)=v(-c^{-1}\log r)$ therefore satisfies $f'(1-)=-f(1)=-1$. This is the slope of the minimizing radial factor in the auxiliary disk, not a Neumann boundary condition on the physical domain.

Finally, $M_a>0$ and the energy identity give
$E(h_a)/M_a=\kappa_a-c/M_a<\kappa_a$.
\end{proof}
Physical energy equals $\kappa M$ for interior two-pole parameters and $\kappa M-c$ for a one-pole boundary parameter. With both parameters on $\partial\D$, the trial function is constant of modulus one and has zero energy.

At coalescence the limiting function is the constant $1$, also with zero energy. For interior two-pole functions normalized by $v(0)=1$, however,
\[
 F_n\longrightarrow1\quad\text{in }L^p(\Omega)\ (1\le p<\infty),
 \qquad E(F_n)\longrightarrow2c=4\pi\ne E(1).
\]
Thus moment continuity, rather than continuity of physical energy, is the relevant property at this stratum.

\begin{lemma}\label{lem:shifted-ritz}\label{lem:ritz-remainder}
Fix a real mass-normalized first positive eigenfunction $\phi_1$.
Let $F=X+iY\in H^1(\Omega;\mathbb C)$ satisfy
$X,Y\perp\Span\{1,\phi_1\}$ and $Q_F=(E/2)I_2$, where $E=E(F)>0$.
Then
\begin{equation}
 \lambda_2^{-1}+\lambda_3^{-1}\ge2M(F)/E(F).
 \label{eq:shifted-ritz}
\end{equation}
Choose a mass-orthonormal eigenbasis in spectral order, placing the fixed
$\phi_1$ at index $1$. The other vectors in its eigenspace remain at
indices $k\ge2$; thus $\lambda_2=\lambda_1$ is allowed.
To record the quantitative remainder used in Section~\ref{sec:quadratic},
put $h_1=\sqrt{2/E}X$, $h_2=\sqrt{2/E}Y$ and
$p_k=\lambda_k\sum_{i=1}^2|\langle h_i,\phi_k\rangle|^2$ for $k\ge2$.
Then $0\le p_k\le1$, $\sum p_k=2$, and
\begin{align}
 \mathscr R(F)&:=\lambda_2^{-1}+\lambda_3^{-1}-2M(F)/E(F)\notag\\
 &=(1-p_2)(\lambda_2^{-1}-\lambda_3^{-1})
 +\sum_{k\ge4}p_k(\lambda_3^{-1}-\lambda_k^{-1})\ge0.
 \label{eq:exact-ritz-remainder}
\end{align}
\end{lemma}
\begin{proof}
The functions $h_1,h_2$ are orthonormal for the Dirichlet inner product on
$\mathcal V=H^1(\Omega)\cap\{1,\phi_1\}^{\perp_{L^2}}$. The bound
$E(v)\ge\lambda_2M(v)$ makes this a Hilbert space with energy-orthonormal
basis $\phi_k/\sqrt{\lambda_k}$, $k\ge2$. Bessel's inequality and Parseval give
\[
 0\le p_k\le1,\qquad \sum_{k\ge2}p_k=2,
 \qquad \frac{2M(F)}{E(F)}=\sum_{k\ge2}\frac{p_k}{\lambda_k}.
\]
Subtracting the last identity gives \eqref{eq:exact-ritz-remainder}, whose
terms are nonnegative. The proof removes only the chosen $\phi_1$ and thus
also covers a multiple first positive eigenvalue.
\end{proof}
This is the two-mode shifted form of Hersch's trace variational principle;
see~\cite{Hersch1961,HileXu1993,CY}.

\Needspace{6\baselineskip}
\begin{proof}[Proof of \cref{thm:main}]
In either setting, Lemma~\ref{lem:LL-form} supplies a real $L^2(\Omega,dv_g)$-normalized eigenfunction $\phi_1$ associated with the first positive Neumann eigenvalue. Put $w=\phi_1\circ\Phi$. Then
\[
 \int_\D w\,d\mu=0,\qquad
 \int_\D|w|\,d\mu
 =\int_\Omega|\phi_1|\,dv_g\le\sqrt A,
\]
so Proposition~\ref{prop:centering} gives a trial function whose real components are mass-orthogonal to $1$ and $\phi_1$.

For an interior pair, Lemma~\ref{lem:energy} and \eqref{eq:positiveallpair} give
\[
 Q_F=\frac{E(F)}2I_2,\qquad E(F)/M(F)=\kappa(H)<\beta.
\]
For the one-pole alternative, Theorem~\ref{thm:transmission} and
kernel continuity in Lemma~\ref{lem:continuity} give $\kappa_a\le\beta$
by approaching the boundary parameter through distinct interior pairs.
Lemma~\ref{lem:onepole-boundary} therefore gives
\[
 Q_{h_a}=\frac{E(h_a)}2I_2,\qquad
 E(h_a)/M_a=\kappa_a-c/M_a<\beta.
\]
Let $U$ be the selected field in either case, and write $M=M(U)$,
$E=E(U)$. Applying Lemma~\ref{lem:shifted-ritz} yields
\[
 \lambda_2^{-1}+\lambda_3^{-1}\ge2M/E>2/\beta.
\]
Since $2/\lambda_2\ge1/\lambda_2+1/\lambda_3$, the scalar bound follows.
\end{proof}

\begin{proof}[Proof of the comparisons in Theorem~\ref{thm:critical}]
For a distinct interior Green pair each branch has area strictly
smaller than $A=4\pi/K$, so $G_K$ is positive throughout both
branch intervals. The local Bol and coarea comparison
\eqref{eq:order} and Proposition~\ref{prop:compression} therefore
apply without change. Proposition~\ref{prop:critical-model} gives
$\kappa(H)\le2K$.
Kernel continuity in Lemma~\ref{lem:continuity} gives the same
bound at a one-pole boundary parameter. Its physical quotient is
$\kappa_a-c/M_a<2K$, by the energy identity in
Lemma~\ref{lem:onepole-boundary}. Here $\beta_K(A/2)=2K$ by
Proposition~\ref{prop:critical-model}; the energy identity itself
does not require the strict area condition.
Apply Proposition~\ref{prop:centering} with the fixed first
positive eigenfunction as weight. In the interior alternative,
the selected field has scalar energy matrix and quotient at most
$2K$; in the one-pole alternative its quotient is strictly below
$2K$. Lemma~\ref{lem:shifted-ritz} gives the reciprocal inequality
in either case. Since
$2/\lambda_2\ge1/\lambda_2+1/\lambda_3$, the scalar inequality
follows.
\end{proof}

\section{Sharpness with exact area and curvature}\label{sec:sharpgeometry}
For the proof of Theorem~\ref{thm:sharp}, we work in the fixed model surface $\mathbb M_K$. The critical-area construction at the end of this section varies the model curvature.
The contact point has zero $H^1$ capacity in dimension two, which
provides cutoffs of vanishing energy for the upper spectral limit.
Restriction to the two moving disks and absence of mass concentration
in their shrinking overlap give the lower limit. Smoothing and a local
area correction then complete the construction.

\subsection{Exact-area construction}
Choose $R$ with $V_K(R)=m$ and a unit-speed geodesic $\gamma$ through $o$. For small $\delta>0$ put
\begin{equation}
 p_\delta^\pm=\gamma(\pm(R-\delta/2)),\qquad d_\delta=2R-\delta.
 \label{eq:centers}
\end{equation}
Put $p_\pm=\gamma(\pm R)$.
The two limiting radius-$R$ disks are tangent only at $o$. On the sphere $R$ is below a hemisphere and $2R<\pi/\sqrt K$, so the tangency is along the unique shorter center geodesic.

\begin{lemma}\label{lem:exactarea}
There is a unique common radius $R_\delta>R$ near $R$ such that
\begin{equation}
 U_\delta=B_K(p_\delta^+,R_\delta)\cup B_K(p_\delta^-,R_\delta),
 \qquad |U_\delta|=A.
 \label{eq:union}
\end{equation}
One has $R_\delta\to R$. These are connected Lipschitz disk-type domains of curvature exactly $K$. Their intersection $I_\delta$ and the portions of the two disk boundaries contained in the other disk shrink to $o$. There are $0\le\eta_\delta\le1$ in $H^1_{\rm loc}(\mathbb M_K)$, vanishing on a neighborhood of those sets, such that
\begin{equation}
 \int_{U_\delta}|\nabla\eta_\delta|^2\longrightarrow0,
 \qquad |\{\eta_\delta\ne1\}\cap U_\delta|\longrightarrow0,
 \qquad \eta_\delta\to1\text{ away from }o.
 \label{eq:cutoff}
\end{equation}
\end{lemma}
\begin{proof}
At radius $R$ and $\delta>0$ the two disks overlap, so their union has area strictly below $A=2V_K(R)$. At $\delta=0$, their closures meet only at $o$ and their union has area $A$. For each sufficiently small $\epsilon>0$, enlarging the common radius to $R+\epsilon$ adds an open set near a boundary point away from $o$. Therefore
\[
 \bigl|B_K(p_+,R+\epsilon)\cup B_K(p_-,R+\epsilon)\bigr|>A.
\]
Keep $R+\epsilon$ below a hemisphere if $K>0$. By continuity of the union area in the centers and radius, the same strict upper comparison holds for all sufficiently small $\delta>0$. The union area is strictly increasing with the common radius, since an increase adds an open outer piece. The intermediate value theorem gives a unique $R_\delta\in(R,R+\epsilon)$ with union area $A$. Since $\epsilon$ is arbitrary, $R_\delta\to R$.

The disks and their intersection are geodesically convex (keep $R+\epsilon$ below a hemisphere if $K>0$). Since $d_\delta=2R-\delta<2R_\delta$ and $d_\delta>0$, the equal-radius boundary circles intersect. In the spherical case, $R<\pi/(2\sqrt K)$ and $R_\delta\to R$, so for small $\delta$ also $2R_\delta<\pi/\sqrt K$; this excludes the complementary/antipodal ambiguity. Hence, in each space form, the two boundary circles meet transversely at exactly two points. Their union is therefore a Jordan domain with two finite-angle corners. Any sequence in the intersection staying away from $o$ would have a limit in both limiting closed disks, contradicting unique tangency. Those interior boundary arcs shrink for the same reason.

Choose $\rho_\delta\downarrow0$ with all these sets in $B_K(o,\rho_\delta/2)$. In fixed normal coordinates take a radial cutoff zero for $r\le\rho_\delta$, one for $r\ge\sqrt{\rho_\delta}$, and linear in $\log r$ between. Its Euclidean energy is $4\pi/|\log\rho_\delta|\to0$. Smooth metric comparability proves \eqref{eq:cutoff}; the exceptional area also vanishes.
\end{proof}

Let $D_\pm$ be abstract copies of the radius-$R$ disk and set
\begin{equation}
 \Lambda_k^0=\lambda_k(D_+\sqcup D_-),\qquad
 \Lambda_0^0=\Lambda_1^0=0,\qquad
 \Lambda_2^0=\cdots=\Lambda_5^0=\beta.
 \label{eq:limitspectrum}
\end{equation}
The multiplicity four comes from two copies of the two-dimensional model eigenspace. Model isometries and radial scaling in exponential coordinates give smooth maps
$T_\delta^\pm:D_\pm\to B_K(p_\delta^\pm,R_\delta)$ whose pulled-back metrics converge smoothly, with uniform metric and Jacobian bounds. These maps compare domains; they do not rescale the metric on the domain.

The overlapping union, its geometric contact limit, and the spectral direct sum are distinguished in Figure~\ref{fig:sharpness}.

\begin{figure}[htbp]
\centering
\begin{tikzpicture}[x=1cm,y=1cm,font=\small,>=Stealth]
  \node[align=center] at (2.35,1.82) {\textbf{(a)} Exact-area union $U_\delta$};
  \node[align=center] at (8.10,1.82) {\textbf{(b)} Geometric contact};
  \node[align=center] at (13.68,1.82) {\textbf{(c)} Spectral direct sum};
  \def\rd{1.01089881059323}
  \pgfmathsetmacro{\ang}{acos(0.9/\rd)}
  \pgfmathsetmacro{\yi}{sqrt(\rd^2-0.9^2)}
  \begin{scope}
    \clip (1.45,0) circle (\rd);
    \fill[gray!18] (3.25,0) circle (\rd);
  \end{scope}
  \draw[thin,densely dotted] (1.45,0) circle (\rd);
  \draw[thin,densely dotted] (3.25,0) circle (\rd);
  \draw[semithick] ($(1.45,0)+(\ang:\rd)$) arc[start angle=\ang,end angle=360-\ang,radius=\rd];
  \draw[semithick] ($(3.25,0)+(180+\ang:\rd)$) arc[start angle=180+\ang,end angle=540-\ang,radius=\rd];
  \fill (2.35,\yi) circle (1.4pt);
  \fill (2.35,-\yi) circle (1.4pt);
  \fill (1.45,0) circle (1.3pt);
  \fill (3.25,0) circle (1.3pt);
  \node[fill=white,inner sep=1pt] at (2.35,0) {$I_\delta$};
  \draw[->] (1.45,0)--($(1.45,0)+(135:\rd)$);
  \node[fill=white,inner sep=1.5pt] at (1.53,0.56) {$R_\delta$};
  \draw[<->] (1.45,-1.44)--node[below] {$2R-\delta$}(3.25,-1.44);
  \node at (2.35,1.35) {$|U_\delta|=A$};
  \draw[semithick] (7.10,0) circle (1.0);
  \draw[semithick] (9.10,0) circle (1.0);
  \draw[dashed] (8.10,0) circle (0.32);
  \fill (8.10,0) circle (1.4pt);
  \node[below=9pt,fill=white,inner sep=1pt] at (8.10,0) {$o$};
  \node at (7.10,0) {$R$};
  \node at (9.10,0) {$R$};
  \node[align=center,font=\footnotesize] at (8.10,-1.43) {shrinking cutoff neighborhood};
  \draw[->] (4.67,0.45)--node[above] {$\delta\downarrow0$}(5.76,0.45);
  \draw[semithick] (12.55,0) circle (0.87);
  \draw[semithick] (14.82,0) circle (0.87);
  \node[align=center] at (12.55,0) {$D_+$\\$A/2$};
  \node[align=center] at (14.82,0) {$D_-$\\$A/2$};
  \node at (13.68,-1.35) {$D_+\sqcup D_-$};
  \draw[->] (2.35,-2.16)--node[fill=white,inner sep=3pt] {fixed-index Neumann spectral limit}(13.68,-2.16);
\end{tikzpicture}
\caption{The sharpness construction in a fixed constant-curvature model (schematic for general $K$, not drawn to a common scale across panels). The union $U_\delta$ has finite-angle corners. Each constituent disk has area $(A+|I_\delta|)/2>A/2$; dotted arcs inside the overlap are not part of $\partial U_\delta$. The dashed circle in (b) indicates a cutoff neighborhood shrinking to $o$, not a hole removed from the domain. The limiting closed disks touch geometrically, whereas the fixed-index spectral limit is the disjoint union of two abstract copies of area $A/2$. Corner smoothing and a local area correction yield $U_n$, which need not remain a union of two exact disks.}
\label{fig:sharpness}
\end{figure}

\subsection{Spectral convergence and smoothing}
\begin{theorem}\label{thm:spectral}
For every fixed $k\ge0$,
\begin{equation}
 \lambda_k(U_\delta)\longrightarrow\Lambda_k^0.
 \label{eq:spectral}
\end{equation}
\end{theorem}
\begin{proof}
\emph{Upper limit.} Transplant the first $k+1$ eigenfunctions of $D_+\sqcup D_-$ by $(T_\delta^\pm)^{-1}$, multiply by $\eta_\delta$, and extend by zero outside each moving disk. They lie in $H^1(U_\delta)$ because the cutoff vanishes near the interior portions of the moving-disk boundaries across which the functions are extended by zero. After the cutoff, trial functions transplanted from opposite components have disjoint supports in the overlap region. Smooth convergence of the metrics and \eqref{eq:cutoff} imply convergence of all finite-dimensional mass and energy matrices: the original eigenfunctions and gradients are bounded, the cutoff energy vanishes, and mixed terms vanish by Cauchy--Schwarz. The mass matrix tends to $I_{k+1}$, so its least eigenvalue is at least $1/2$ for sufficiently small $\delta$; the transplanted space retains dimension $k+1$. Matrix convergence controls the quotient uniformly on its coefficient unit sphere. Therefore
\begin{equation}
 \limsup_{\delta\downarrow0}\lambda_k(U_\delta)\le\Lambda_k^0.
 \label{eq:limsup}
\end{equation}

\emph{Compactness and mass for the lower limit.} Choose $\delta_n\downarrow0$ such that $\lambda_k(U_{\delta_n})\to\liminf_{\delta\downarrow0}\lambda_k(U_\delta)$, and choose $L^2$-orthonormal eigenfunctions $u_{0,n},\ldots,u_{k,n}$ on $U_{\delta_n}$ corresponding to the first $k+1$ eigenvalues. Restrict to each whole moving disk and pull back to $D_\pm$. The upper bound supplies uniform $H^1$ bounds. A common subsequence for these finitely many functions converges weakly in $H^1$ and strongly in $L^2$ on each fixed disk. No mass remains in the shrinking overlap: the uniform Sobolev inequality on either smoothly varying disk gives
\begin{equation}
 \int_{I_{\delta_n}}|u|^2
 \le |I_{\delta_n}|^{1/2}\|u\|_{L^4}^2
 \le C|I_{\delta_n}|^{1/2}\|u\|_{H^1}^2\longrightarrow0.
 \label{eq:noconcentration}
\end{equation}
Products are controlled by Cauchy--Schwarz. Since a union integral is the sum of disk integrals minus the intersection integral, the limiting pairs $u_\ell=(u_\ell^+,u_\ell^-)$ are orthonormal in $L^2(D_+\sqcup D_-)$.

\emph{Energy without double counting.} First remove a fixed small contact neighborhood from each limiting disk. The images of the two remaining regions are disjoint for large $n$. Pull each such truncated image back by $T_{\delta_n}^\pm$ to the corresponding fixed truncated subset of $D_\pm$; the pulled-back metrics converge smoothly there. For each fixed coefficient vector $b$, weak lower semicontinuity on these fixed truncated domains gives the corresponding inequality for the pair. Letting the contact neighborhoods shrink and using monotone convergence then gives
\begin{equation}
 \sum_\pm\int_{D_\pm}\left|\nabla\sum_{\ell=0}^k b_\ell u_\ell^\pm\right|^2
 \le\liminf_n\int_{U_{\delta_n}}\left|\nabla\sum_{\ell=0}^k b_\ell u_{\ell,n}\right|^2
 \le(\liminf_n\lambda_k(U_{\delta_n}))|b|^2.
 \label{eq:energy-liminf}
\end{equation}
Min--max on the $(k+1)$-dimensional limiting space proves the opposite inequality to \eqref{eq:limsup}. This proves \eqref{eq:spectral}.
\end{proof}
The two full-disk restrictions account for all mass. This fixed-index limit is specific to their covering construction: arbitrary thin connectors can carry additional modes, and no uniformity in the index is claimed.

\begin{lemma}\label{lem:mapstability}
Let $W\Subset S$ be Lipschitz. Suppose $W_j=\Psi_j(W)$, where the maps and inverses are uniformly bi-Lipschitz on a common neighborhood, all images remain in a relatively compact smooth metric neighborhood, $\Psi_j\to\mathrm{id}$ uniformly and $D\Psi_j\to I$ almost everywhere in fixed charts. Then $|W_j|\to|W|$ and $\lambda_k(W_j)\to\lambda_k(W)$ for each fixed $k$.
For every fixed $W=U_\delta$, such smoothing maps, followed by a local area correction, give smooth disk-type domains with area exactly $A$, curvature exactly $K$, and the same fixed-index spectral limits.
\end{lemma}
\begin{proof}
Pull the forms back to $W$. In local coordinates put $J_j=D\Psi_j$ and $g_j=g\circ\Psi_j$. The mass density and energy tensor are
\[
 b_j=\sqrt{\det g_j}\,|\det J_j|,\qquad
 A_j=b_jJ_j^{-1}g_j^{-1}J_j^{-T}.
\]
The stated uniform bi-Lipschitz and compact-neighborhood hypotheses bound $b_j$ above and below and make $A_j$ uniformly elliptic; both converge almost everywhere to the original coefficients. Finite trial spaces give the upper spectral limit by dominated convergence. For the lower limit, choose $L^2$-normalized eigenfunctions among the first finitely many eigenvalues. The upper spectral bound gives uniformly bounded pulled-back $H^1$ norms, hence weak $H^1$ and strong $L^2$ limits. Let $b$ and $\mathsf A$ denote the corresponding limiting mass density and energy tensor in a fixed chart. For a fixed nonnegative partition-of-unity function $\chi$ supported in the chart, strong $L^2$ convergence of $v_j,z_j$ and the boundedness of $b_j$ give
\[
\begin{aligned}
 \left|\int\chi b_jv_jz_j\,dx-\int\chi bvz\,dx\right|
 &\le C\bigl(\|v_j-v\|_2\|z_j\|_2+\|v\|_2\|z_j-z\|_2\bigr)\\
 &\quad+\int\chi|b_j-b|\,|vz|\,dx\longrightarrow0.
\end{aligned}
\]
The two terms vanish by strong $L^2$ convergence and dominated convergence for the fixed function $\chi|vz|$, respectively. Summing over a finite atlas preserves mass orthonormality, without pointwise bounds on the varying eigenfunctions.

For the energy lower limit, fix an $L^2$ vector field $Z$ in a chart. Uniform ellipticity and almost-everywhere convergence imply $A_j^{1/2}Z\to\mathsf A^{1/2}Z$ strongly in $L^2$ by dominated convergence; the square-root map is continuous on the common compact ellipticity range. Pairing with $\nabla v_j\rightharpoonup\nabla v$ gives $A_j^{1/2}\nabla v_j\rightharpoonup\mathsf A^{1/2}\nabla v$ in $L^2$. Localize by the square roots of the same partition-of-unity functions and apply weak lower semicontinuity, then sum the finitely many contributions. The same min--max argument as above gives the spectral lower limit. Area convergence follows from the bounded densities $b_j\to b$ almost everywhere.

For a fixed $U_\delta$, represent each corner locally by a Lipschitz graph $y=\varphi(x)$, smooth away from the corner. Mollify and blend it to smooth graphs $\varphi_j$ with bounded slopes, uniform convergence, and $\varphi_j'\to\varphi'$ almost everywhere. Put $d_j=\varphi_j-\varphi$. With a smooth vertical cutoff $\zeta$ equal to one at zero, the local map
\[
 (x,y)\mapsto(x,y+\zeta(y-\varphi(x))d_j(x))
\]
sends the old graph to the new one. For small $\|d_j\|_\infty$ its vertical derivative stays positive; it and its inverse are uniformly Lipschitz for this fixed corner, and the derivative tends almost everywhere to $I$. Disjoint corner charts give smooth domains $\widetilde W_j$ satisfying the first assertion.

Choose a smooth boundary arc away from these charts. Let $V$ be a small neighborhood of that arc on which every $\widetilde W_j$ coincides with $W$, and choose a smooth vector field $X$ supported in $V$ whose boundary flux across the arc is positive. Let $\Theta_t$ be its local flow; in particular $\Theta_t=\mathrm{id}$ outside $V$. Since $\widetilde W_j\cap V=W\cap V$ for every $j$, while the flow is the identity off $V$, the area increment is the same as for the reference domain:
\[
 |\Theta_t(\widetilde W_j)|-|\widetilde W_j|
 =|\Theta_t(W)|-|W|=:q(t).
\]
Here $q(0)=0$ and $q'(0)>0$. Set $t_j=q^{-1}(A-|\widetilde W_j|)\to0$. The corrected domains are smooth, have exact area $A$, and their composite maps retain the spectral hypotheses. All changes move domains inside the fixed $\mathbb M_K$, so curvature remains exactly $K$.
\end{proof}

\subsection{Completion of sharpness}
\begin{proof}[Proof of \cref{thm:sharp}]
Choose any sufficiently small $\delta_n\downarrow0$. For each fixed $n$, choose a smoothing and area correction of $U_{\delta_n}$ fine enough that
\[
 |U_n|=A,\qquad
 |\lambda_k(U_n)-\lambda_k(U_{\delta_n})|<1/n\quad(0\le k\le n).
\]
Theorem~\ref{thm:spectral} gives the required limits for every fixed $k$. The smoothing is chosen separately for each $n$. In particular $\lambda_1\to0$ and $\lambda_2,\lambda_3\to\beta>0$, so the scalar supremum and reciprocal infimum in \eqref{eq:sharp-both} follow from Theorem~\ref{thm:main}. Their strict inequalities exclude attainment in either connected class.

The sequence has a representation in setting (L) as well. For $K=0$ use the planar chart, and for $K<0$ use the Poincar\'e disk chart on a relatively compact neighborhood of each $\overline U_n$. For $K>0$, the limiting union has area $A<4\pi/K$, so its complement contains an open set. Choose a stereographic projection point there; the construction and its local smoothing can be taken to avoid a fixed neighborhood of this point for all sufficiently large $n$. In each case the chart image is a bounded smooth simply connected planar domain and the model metric has a smooth positive conformal weight on a neighborhood of its closure. It therefore satisfies \eqref{eq:LL-weight}. These chart maps preserve area and spectrum. The construction gives an extremizing sequence in each class; it does not classify all such sequences.
\end{proof}

\paragraph{Critical-area sharpness.}
\begin{proof}[Proof of sharpness in Theorem~\ref{thm:critical}]
Fix $A=4\pi/K$, $m=A/2$, and choose $0<K_n<K$ with
$K_n\uparrow K$. Then $K_nA<4\pi$, so Theorem~\ref{thm:sharp}
applies at the parameters $(K_n,A)$. For each $n$, choose a
smooth connected domain $V_n\subset\mathbb M_{K_n}$ from that
sequence such that
\[
 |V_n|=A,\qquad
 \lambda_1(V_n)<1/n,\qquad
 |\lambda_j(V_n)-\beta_{K_n}(m)|<1/n\quad(j=2,3).
\]
Each domain has constant curvature $K_n\le K$.
The conformal model densities $q_{K_n,m}$, defined in \eqref{eq:curvature-density}, converge uniformly on
$\overline\D$ to $q_{K,m}$ and are uniformly bounded above and
below by positive constants. Their Dirichlet energies agree, and
their mass forms converge uniformly relative to the limiting
mass form. The full min--max principle therefore gives
$\beta_{K_n}(m)\to\beta_K(m)=2K$.
This proves the stated spectral limits in setting \textup{(G)}.
The chart representation in the proof of Theorem~\ref{thm:sharp}
also places every $V_n$ in setting \textup{(L)}. The comparison
inequalities already proved then identify the supremum and
infimum. This construction preserves the fixed curvature upper
bound $K$; its curvatures are $K_n$, rather than identically $K$.
\end{proof}

\section{Quantitative model and domain deficits}\label{sec:quadratic}
We estimate the spectral, area-split, and kernel deficits for an interior two-pole selection, then prove coercivity for time profiles. These analytic estimates concern the selected configuration, not global geometric stability.
Proposition~\ref{prop:global-split} gives the global quadratic loss
caused by an unequal model area split. Its proof combines local
nondegeneracy, strict endpoint bounds, and compactness. Section~\ref{sec:interior-deficits}
combines this loss with the spectral remainder and kernel increment;
Proposition~\ref{prop:uniformkernel} gives a separate coercivity estimate
in profile space.

\subsection{Cumulative-area comparison}\label{sec:densityfold}
We record the cumulative-area comparison used in the coercivity and symmetry estimates.
The model density and cumulative area are given by
\eqref{eq:curvature-density} and \eqref{eq:Acum-model}.

\begin{lemma}\label{lem:cumulative}
Suppose $h>0$ is continuous on $\D$, $\int_\D h=s<\infty$, and $\Delta\log h+2\varkappa h\ge0$ in $\mathcal D'(\D)$, and $\varkappa s<4\pi$ if $\varkappa>0$. Then
\begin{equation}
 \mathcal A_h(r):=\int_{|z|<r}h\le\mathcal A_{\varkappa,s}(r).
 \label{eq:cumulative-order}
\end{equation}
For every nonnegative nondecreasing absolutely continuous $f$ on $[0,1]$,
\begin{equation}
 \int_\D f^2(h-q_{\varkappa,s})
 =\int_0^1(f^2)'(\mathcal A_{\varkappa,s}-\mathcal A_h)\,dr\ge0.
 \label{eq:radialmass-order}
\end{equation}
If $(f^2)'>0$ almost everywhere, equality implies $h=q_{\varkappa,s}$.
\end{lemma}
\begin{proof}
On each compact coordinate disk, Lemma~\ref{lem:local-bol} and Cauchy--Schwarz give
\[
 \mathcal A_h(4\pi-\varkappa\mathcal A_h)
 \le L_h(r)^2\le2\pi r\mathcal A_h'(r).
\]
Thus $(\log[\mathcal A_h/(4\pi-\varkappa\mathcal A_h)])'\ge2/r$.
Integrate from $r$ to $r_1<1$ and let $r_1\uparrow1$, using the total mass $s$, to obtain \eqref{eq:cumulative-order}. Fubini gives $\mathcal A_h\in W^{1,1}(0,1)$ with endpoint traces $0,s$, as for the model. Since $f^2$ is absolutely continuous and bounded, integration by parts gives \eqref{eq:radialmass-order}; the equal total masses cancel the endpoint terms. In the equality case put $D(r)=\mathcal A_{\varkappa,s}(r)-\mathcal A_h(r)$. Then $D\ge0$ is continuous and
$\int_0^1(f^2)'D=0$. Since $(f^2)'>0$ almost everywhere, $D$ cannot be positive at any point: otherwise continuity would make it positive on an interval of positive measure. Hence
$\mathcal A_h=\mathcal A_{\varkappa,s}$ on $[0,1]$, and therefore
$\mathcal A_h'=\mathcal A_{\varkappa,s}'$ almost everywhere. Comparing the two endpoints of
\[
 \mathcal A_h(4\pi-\varkappa\mathcal A_h)
 \le L_h(r)^2\le2\pi r\mathcal A_h'(r)
\]
with the model identity shows that both inequalities are equalities for almost every $r$. Equality in the circular Cauchy--Schwarz inequality makes $h(re^{i\theta})$ independent of $\theta$ for almost every $r$ and almost every $\theta$. For each such radius, continuity in $\theta$ upgrades this to constancy on the whole circle. Approximating any radius by these full-measure radii and using continuity
shows that $h$ is radial throughout $\D$. Thus
$2\pi rh(r)=\mathcal A_h'(r)=\mathcal A_{\varkappa,s}'(r)$, and continuity gives
$h=q_{\varkappa,s}$ pointwise. This rigidity uses the curvature hypothesis; equality of cumulative masses alone need not make a density radial.
\end{proof}

\subsection{Spectral remainder and high-mode energy}
For the field and weights in Lemma~\ref{lem:shifted-ritz}, define
\begin{equation}
 \Delta_{\rm sp}=\lambda_2^{-1}+\lambda_3^{-1}-2/\beta.
 \label{eq:remainder-deficit}
\end{equation}
If $E/M<\beta$, then $0\le\mathscr R(F)<\Delta_{\rm sp}$ since
$\Delta_{\rm sp}-\mathscr R(F)=2M/E-2/\beta>0$. For $r\ge3$ with
$\lambda_{r+1}>\lambda_3$, the nonnegative terms in
\eqref{eq:exact-ritz-remainder} give
\begin{equation}
 \sum_{k\ge r+1}p_k\le
 \frac{\mathscr R(F)}{\lambda_3^{-1}-\lambda_{r+1}^{-1}}.
 \label{eq:spectral-leakage}
\end{equation}
The $p_k$ are spectral energy weights, not $L^2$ mass fractions. The equality condition is independent of the basis in a repeated eigenspace. In the restricted space $\mathcal V$, put
\[
 E_{<}=\Span\{\phi_k:k\ge2,\ \lambda_k<\lambda_3\},\qquad
 E_{\le}=\Span\{\phi_k:k\ge2,\ \lambda_k\le\lambda_3\}.
\]
Then $\mathscr R(F)=0$ exactly when $E_{<}\subset\Span\{X,Y\}\subset E_{\le}$. Indeed every weight above $\lambda_3$ must vanish, and if $\lambda_2<\lambda_3$, the additional condition $p_2=1$ says $\phi_2$ lies in the trial plane. The converse follows from the same identity. A repeated $\lambda_3$ does not require the plane to contain its whole eigenspace. The tail estimate keeps its positive spectral-gap denominator; a small remainder alone does not give a domain-uniform small tail.

\subsection{Model area-split deficits}

\begin{lemma}\label{lem:analytic-regular-family}
Fix $K\in\R$ and fix the positive leading coefficient of $v_\lambda(a)/\sqrt a$ independently of $\lambda$. For every $\lambda_0>0$ and $0<s_0<a_D(\lambda_0)$, the functions $v_\lambda(s)$ and $R_\lambda(s)$ are real-analytic in $(\lambda,s)$ near $(\lambda_0,s_0)$. In particular, this applies near $(\beta,m)$.
\end{lemma}
\begin{proof}
Multiplying \eqref{eq:modelODE} by $G_K$ gives
\[
 G_K^2v''+G_KG_K'v'+(\lambda G_K-c^2)v=0.
\]
Since $G_K=a(2c-Ka)$, the center is a regular singular point with indicial roots $\pm\tfrac12$. For the regular root write
\[
 v_\lambda(a)=a^{1/2}\sum_{n=0}^\infty p_n(\lambda)a^n,
 \qquad p_0>0\ \text{fixed},\quad p_{-1}=0.
\]
Comparing coefficients gives, for $n\ge1$,
\[
\begin{aligned}
 4c^2n(n+1)p_n
 &+2c\bigl[\lambda-Kn(2n-1)\bigr]p_{n-1}\\
 &+\bigl[K^2((n-1)^2-\tfrac14)-K\lambda\bigr]p_{n-2}=0.
\end{aligned}
\]
Thus each $p_n$ is a polynomial in $\lambda$. On a complex parameter disk $|\lambda|\le L$, the recurrence yields
\[
 |p_n|\le C_1|p_{n-1}|+C_2|p_{n-2}|,\qquad
 C_1=\frac{L}{4c}+\frac{|K|}{c},\quad
 C_2=\frac{K^2+|K|L}{4c^2}.
\]
Choose $Q\ge\max\{1,2C_1,\sqrt{2C_2}\}$. Induction gives $|p_n(\lambda)|\le p_0Q^n$, so the series after removal of the factor $a^{1/2}$ converges jointly and locally uniformly for $|a|<Q^{-1}$ and $|\lambda|\le L$. Taking the bound on a larger parameter disk and applying Cauchy's estimate also controls its parameter derivatives. Since $p_0$ is independent of $\lambda$,
\[
 \partial_\lambda v_\lambda=O(a^{3/2}),\qquad
 \partial_\lambda v_\lambda'=O(a^{1/2}),
\]
locally uniformly in the spectral parameter.

Choose a small positive $a_*$ in this common interval. The two initial values at $a_*$ are analytic in $\lambda$. Analytic dependence for the nonsingular ordinary equation then continues the solution to a neighborhood of $(\lambda_0,s_0)$. Since $v_{\lambda_0}(s_0)>0$, division in \eqref{eq:response} is valid nearby. This proves the asserted joint analyticity and the center estimates used below.
\end{proof}

\begin{theorem}\label{thm:spectral-response}
For every $K\in\mathbb R$, $\lambda>0$, and $0<s<a_D(\lambda)$,
\begin{equation}
 \partial_\lambda R_\lambda(s)
 =-\frac{\int_0^s v_\lambda^2\,da}{v_\lambda(s)^2}<0.
 \label{eq:freqderiv}
\end{equation}
Rescale the regular family by one fixed positive factor so that $v_\beta(m)=1$, and write $M_0=\int_0^m v_\beta^2\,da$, consistently with \eqref{eq:modelenergy}. Then, for fixed $K\in\mathbb R$, $A>0$ with $KA<4\pi$ if $K>0$, and sufficiently small $|\varepsilon|<m$,
\begin{equation}
 \kappa_K(m+\varepsilon,m-\varepsilon)
 =\beta-C_{K,A}\varepsilon^2+O(\varepsilon^4),\qquad
 C_{K,A}=\frac{c^2G_K'(m)}{2M_0G_K(m)^2}>0.
 \label{eq:quadratic}
\end{equation}
No uniform remainder at $KA\uparrow4\pi$ is asserted.
\end{theorem}
\begin{proof}
Keep the leading center coefficient of $v_\lambda$ independent of $\lambda$. Lemma~\ref{lem:analytic-regular-family} justifies differentiating the regular family and its center expansion. With $\dot v=\partial_\lambda v_\lambda$, subtraction of the differentiated equation gives
\[
 [G_K(v_\lambda\dot v'-v_\lambda'\dot v)]'=-v_\lambda^2.
\]
The center term is zero, proving \eqref{eq:freqderiv}. Differentiating \eqref{eq:riccati} at $R_\beta(m)=0$ also gives
\begin{equation}
 \partial_\lambda R_\beta(m)=-M_0,\qquad
 R_\beta''(m)=-c^2G_K'(m)/G_K(m)^2.
 \label{eq:derivs-at-m}
\end{equation}
Choose $\delta>0$ with $m+\delta<a_D(\beta)$. The analytic dependence is supplied by Lemma~\ref{lem:analytic-regular-family}. For spectral parameters near $\beta$, the fixed positive leading coefficient ensures positivity on a common small center interval; uniform convergence on its complement in $(0,m+\delta]$ ensures positivity there too. Thus both nearby branch profiles remain positive throughout.

Join the two branch solutions, normalized to common interface value $1$, as a whole-line profile $U$; let $\rho_{\rm mod}$ be the pushforward of the two area measures by their signed model Green-time coordinates. Its interface derivatives are $U'(0+)=-R_\lambda(m+\varepsilon)$ and $U'(0-)=R_\lambda(m-\varepsilon)$. Integration by parts on the two half-lines gives, for $\zeta\in\mathcal H$,
\[
 \langle U,\zeta\rangle_{\mathcal H}-\lambda\int U\zeta\,d\rho_{\rm mod}
 =[R_\lambda(m+\varepsilon)+R_\lambda(m-\varepsilon)]\zeta(0).
\]
The center decay removes the infinite-end terms. Thus $U$ is a global weak eigenprofile exactly when
\begin{equation}
 R_\lambda(m+\varepsilon)+R_\lambda(m-\varepsilon)=0.
 \label{eq:matchingroot}
\end{equation}
At $(\beta,0)$ the derivative with respect to the spectral parameter is $-2M_0$. The implicit function theorem gives an analytic even root $\lambda(\varepsilon)$; the whole-line weak equation and positivity identify it with $\kappa_K$ by Lemma~\ref{lem:positive-ground-state}. Evenness first gives $\lambda(\varepsilon)-\beta=O(\varepsilon^2)$, so Taylor expansion, including the quadratic error in the spectral parameter, gives
$0=-2M_0(\lambda(\varepsilon)-\beta)+R_\beta''(m)\varepsilon^2+O(\varepsilon^4)$.
This proves \eqref{eq:quadratic}. Finally $G_K'(m)=4\pi-2Km>0$ in the admitted half-area range.
\end{proof}

\paragraph{Flat specialization.}
Let $j=j'_{1,1}$ and $b_{\rm E}=j^2$. For $K=0$,
$v_\beta(a)=J_1(j\sqrt{a/m})/J_1(j)$ and $\beta=\pi b_{\rm E}/m$.
The Bessel recurrence relations give
$J_0(j)=J_2(j)=J_1(j)/j$, while differentiation of
$\tfrac12 r^2[J_1(jr)^2-J_0(jr)J_2(jr)]$ gives $rJ_1(jr)^2$.
Changing variables $a=mr^2$ therefore yields
\[
 M_0=\frac{2m}{J_1(j)^2}\int_0^1 rJ_1(jr)^2\,dr
 =m\left(1-\frac1{b_{\rm E}}\right).
\]
Substitution in \eqref{eq:quadratic} gives, for fixed $m>0$,
\begin{equation}
 \kappa_0(m+\varepsilon,m-\varepsilon)
 =\frac{\pi b_{\rm E}}m
 \left[1-\frac1{2(b_{\rm E}-1)}
 \left(\frac{\varepsilon}{m}\right)^2
 +O\!\left(\left(\frac{\varepsilon}{m}\right)^4\right)\right].
 \label{eq:flat-quadratic-specialization}
\end{equation}
This positive flat coefficient is distinct from the negative-curvature global coefficient below, which vanishes as $K\uparrow0$.

\begin{proposition}\label{prop:global-split}
Fix $K\in\mathbb R$ and $A>0$, with $KA<4\pi$ if $K>0$,
and put $m=A/2$, $\beta=\beta_K(m)$. There is a constant
$c_{K,A}>0$ such that
\begin{equation}
 \beta-\kappa_K(s,A-s)\ge c_{K,A}(s-m)^2
 \qquad(0<s<A).
 \label{eq:global-quadratic-split}
\end{equation}
The constant is for these fixed parameters; no uniform lower
bound as $KA\uparrow4\pi$ is claimed.
\end{proposition}
\begin{proof}
\emph{Continuity at the endpoint splits.} Extend the coupled value to $s=0,A$ by retaining the
whole-line energy and assigning zero mass to the empty branch.
We show continuity on $[0,A]$. Put $s_+=s$, $s_-=A-s$ and,
for $0<x<1$, write
\[
 t_b(x)=T_{K,b}(bx)
 =\frac1{4\pi}\log\frac{4\pi-Kbx}{x(4\pi-Kb)}
 \quad(0\le b\le A),
\]
where at $b=0$ the displayed expression defines the continuous
extension. On the fixed space
$L^2(0,1)\oplus L^2(0,1)$ use the kernel blocks
\[
 \mathsf K_{ij,s}(x,y)
 =\sqrt{s_is_j}\,
 R_c(\epsilon_it_{s_i}(x),\epsilon_jt_{s_j}(y)),
 \qquad \epsilon_+=1,\quad\epsilon_-=-1.
\]
For positive branch masses this is the unitary transform of the
model kernel operator under $a=s_ix$; at an empty branch its
blocks vanish. The kernels are bounded and converge pointwise
when $s$ varies, so dominated convergence gives Hilbert--Schmidt,
hence operator-norm, continuity. The largest eigenvalue is
positive also at the endpoints, because the remaining branch
has mass $A$ and strictly positive kernel. Its reciprocal is
therefore a continuous extension of $\kappa_K(s,A-s)$.

\emph{Strict endpoint bounds.} We show that the two endpoint values are strictly below
$\beta$. Reflection makes them equal, so take the branch of
area $A$ on the positive half-line. If $K<0$, use the even
profile $g_0(|t|)$. Strict concavity of $a_{K,b}(t)$ in $b$,
including $a_{K,0}=0$, gives
$a_{K,A}(t)<2a_{K,m}(t)$ for $t>0$. The calculation in
\eqref{eq:negative-model-mass} gives its energy $2\beta M_0$
and mass
\[
 2M_0+\int_0^\infty[2a_{K,m}(t)-a_{K,A}(t)]w_0(t)\,dt
 >2M_0.
\]
Thus its quotient is strictly below $\beta$.

Suppose $K\ge0$ and let $s_D=a_D(\beta)$. If $A<s_D$, the
strict concavity from Lemma~\ref{lem:concavity}, with the
continuous center value $R_\beta(0)=c$, gives
$R_\beta(A)+c<2R_\beta(m)=0$.
Use $f(a)=v_\beta(a)/v_\beta(A)$ on the positive branch and
$U(t)=e^{ct}$ on the massless negative half-line. The common
trace is $1$, the negative tail energy is $c$, and integration
of the model equation gives
\[
 \mathcal N(U)-\beta D(U)=R_\beta(A)+c<0.
\]
If instead $A\ge s_D$, then $s_D$ is an actual interior
Dirichlet zero: the right endpoint of the model area interval
is strictly larger than $A$. Use $v_\beta$ on $(0,s_D)$,
extend it by zero to $A$, and use zero on the negative
half-line. The resulting profile belongs to
$H^1(\mathbb R)$ and has quotient $\beta$ by the same center
and endpoint calculation used in Proposition~\ref{prop:modelsplit}.
It cannot minimize, since it vanishes throughout the negative
half-line, whereas Proposition~\ref{prop:kernel} makes the
minimizing profile strictly positive at every finite time.
The endpoint coupled value is therefore again strictly below
$\beta$.

\emph{The global quadratic bound.} For $s\ne m$ define
$Q(s)=[\beta-\kappa_K(s,A-s)]/(s-m)^2$ and set
$Q(m)=C_{K,A}$ from Theorem~\ref{thm:spectral-response}.
The local quadratic expansion makes $Q$ continuous at $m$.
Proposition~\ref{prop:modelsplit}, the strict endpoint bounds,
and $C_{K,A}>0$ show that $Q$ is positive on the compact
interval $[0,A]$. Its positive minimum is an admissible
$c_{K,A}$, proving \eqref{eq:global-quadratic-split}.
\end{proof}

\begin{proposition}\label{prop:negative-global}
\emph{(i) Model estimates.} Let $K<0$ and set $b=4\pi$, $\vartheta=e^{-bt}$, $D=b-Km(1-\vartheta)$ and
\[
 \mathcal J_{K,A}=\int_0^\infty
 \frac{2b^2(-K)\vartheta(1-\vartheta)}{D^3}w_0\,dt.
\]
Then $0<\mathcal J_{K,A}\le(-K)/(2b)$ and, for every $|\varepsilon|<m$,
\begin{align}
 \kappa_K(m+\varepsilon,m-\varepsilon)
 &\le\frac{\beta}{1+\mathcal J_{K,A}\varepsilon^2/(2M_0)},\label{eq:negative-global-value}\\
 \beta-\kappa_K(m+\varepsilon,m-\varepsilon)
 &\ge\frac{\beta\mathcal J_{K,A}\varepsilon^2}
 {2M_0+\mathcal J_{K,A}\varepsilon^2}.\label{eq:negative-global-deficit}
\end{align}
\emph{(ii) Domain estimate.} Independently, let $K\le0$ and take any distinct pair of interior Green poles on a domain in setting (L) or (G) with $K_g\le K$. Let $\alpha=a_++a_-$, $T=\alpha^{-1}$ and $T_*(a)=T_{K,m}(a/2)$. Then
\begin{equation}
 \mathcal L=\int_0^\infty(2a_0-\alpha)w_0\,dt
 =\int_0^A\int_{T(a)}^{T_*(a)}w_0(t)\,dt\,da>0,
 \qquad \kappa(H)\le\frac{2\beta M_0}{2M_0+\mathcal L}.
 \label{eq:negativegain}
\end{equation}
\end{proposition}
\begin{proof}
\emph{Proof of (i).} Put $\delta_K(t)=(-K)(1-\vartheta)$. Formula~\eqref{eq:model-arbitrary-superlevel} gives
\begin{equation}
 2a_{K,m}-a_{K,m+\varepsilon}-a_{K,m-\varepsilon}
 =\frac{2b^2\vartheta \delta_K(t)\varepsilon^2}{D(D^2-\delta_K(t)^2\varepsilon^2)}.
 \label{eq:negative-exact-split-defect}
\end{equation}
Since $D=b+m\delta_K(t)$ and $|\varepsilon|<m$,
\[
 D^2-\delta_K(t)^2\varepsilon^2
 =(b+\delta_K(t)(m-\varepsilon))(b+\delta_K(t)(m+\varepsilon))>0.
\]
Moreover $D^2-\delta_K(t)^2\varepsilon^2\le D^2$, so
$D(D^2-\delta_K(t)^2\varepsilon^2)\le D^3$. Hence the right-hand side of \eqref{eq:negative-exact-split-defect} is bounded below by
$2b^2\vartheta \delta_K(t)\varepsilon^2/D^3$. Integrating against $w_0$ gives
$\mathcal L_K\ge\mathcal J_{K,A}\varepsilon^2$; \eqref{eq:negative-model-test} proves the two bounds. Also $D\ge b$, $\vartheta(1-\vartheta)\le1/4$ and $\int w_0=1$ give
$\mathcal J_{K,A}\le(-K)/(2b)$.

\emph{Proof of (ii).} Write $G(a)=B_+(t)+B_-(t)$ at $a=\alpha(t)$ and $\Delta(a)=a_+(t)-a_-(t)$. Since
\[
 a_+^2+a_-^2=\tfrac12(a^2+\Delta^2),
 \qquad \sum_{j=\pm}G_K(a_j)=a(4\pi-Ka/2)-\tfrac K2\Delta^2,
\]
adding \eqref{eq:order} gives
\begin{equation}
 G(a)-a(4\pi-Ka/2)=\mathcal S(a)-\frac K2\Delta(a)^2\ge0,
 \quad \mathcal S=\sum_j[B_j-G_K(a_j)]\ge0.
 \label{eq:negativetotal}
\end{equation}
Since $\alpha'(t)=-G(\alpha(t))$ and $\alpha(0+)=A$, its inverse satisfies
\[
 T(a)=\int_a^A\frac{d\sigma}{G(\sigma)},\qquad
 T_*(a)=\int_a^A\frac{d\sigma}{2G_K(\sigma/2)},\qquad
 2G_K(\sigma/2)=\sigma(4\pi-K\sigma/2).
\]
The coefficient inequality \eqref{eq:negativetotal} therefore gives $T\le T_*$; inversion of these decreasing functions gives $\alpha\le2a_0$. Theorem~\ref{thm:transmission} supplies an interior interval of strict coefficient excess. The difference $T_*-T$ is positive when the integration range contains that interval, which is the strictness needed below. The branch integration by parts in \eqref{eq:negative-model-mass} yields
$D_H(g_0(|\cdot|))=2M_0+\mathcal L$ and energy $2\beta M_0$.
The equivalence $T(a)<t<T_*(a)\iff\alpha(t)<a<2a_0(t)$ and Tonelli give the double-integral formula, bounded by $A\int_0^\infty w_0=A$. The strict interval and $w_0>0$ make it positive.
\end{proof}
The gain $\mathcal L$ concerns the weighted mass of the fixed trial profile; the area remains $A$. For fixed $K<0,A$, the model estimate holds for all $|\varepsilon|<m$. Its coefficient need not be locally optimal and vanishes as $K\uparrow0$, whereas the flat local coefficient in \eqref{eq:quadratic} is positive.

\subsection{Deficits for an interior two-pole selection}\label{sec:interior-deficits}
For an interior pair selected in the main proof, with the corresponding trial function $F$ built from the minimizing profile and satisfying the two moment conditions, set
$\widehat\kappa=\kappa_K(s_+,s_-)$ and use the kernel increment $d>0$ of \eqref{eq:kernelincrement}. Then
\begin{align}
 \beta-\lambda_2
 &\ge(\beta-\widehat\kappa)+
 \frac{\widehat\kappa^2d}{1+\widehat\kappa d},\label{eq:twogaps}\\
 \Delta_{\rm sp}
 &\ge\mathscr R(F)+2(\widehat\kappa^{-1}-\beta^{-1})+2d>0.\label{eq:reciprocal-two-gaps}
\end{align}
Indeed, writing $F=X+iY$, the two moment conditions in the main proof give
$X,Y\perp_{L^2}\{1,\phi_1\}$. The restricted min--max principle yields
\begin{equation}
 \lambda_2\int_\Omega X^2\,dv_g\le E(X),\qquad
 \lambda_2\int_\Omega Y^2\,dv_g\le E(Y),\qquad
 \lambda_2 M(F)\le E(F).
 \label{eq:shifted-scalar-dependence}
\end{equation}
Thus $\lambda_2\le E/M=\kappa(H)$. From \eqref{eq:compression},
\[
 \kappa(H)\le\frac{\widehat\kappa}{1+\widehat\kappa d},
 \qquad
 \frac1{\kappa(H)}\ge\frac1{\widehat\kappa}+d.
\]
Subtracting the first bound from $\beta$ gives \eqref{eq:twogaps}; substituting the second into
$\Delta_{\rm sp}=\mathscr R(F)+2/\kappa(H)-2/\beta$ gives
\eqref{eq:reciprocal-two-gaps}. The scalar estimate here uses the shifted orthogonality, not energy isotropy alone. For $K\le0$, \eqref{eq:negativegain} alternatively gives
\begin{equation}
 \Delta_{\rm sp}\ge\mathscr R(F)+\mathcal L/(\beta M_0)>0.
 \label{eq:reciprocal-negative-gap}
\end{equation}
For $K<0$, with $\varepsilon=(s_+-s_-)/2$, the model bound gives
\begin{equation}
 \Delta_{\rm sp}\ge\mathscr R(F)+
 \frac{\mathcal J_{K,A}\varepsilon^2}{\beta M_0}+2d.
 \label{eq:negative-combined-global}
\end{equation}
For all admitted curvature signs, Proposition~\ref{prop:global-split}
also gives, with $\varepsilon=(s_+-s_-)/2$,
\begin{equation}
 \Delta_{\rm sp}\ge\mathscr R(F)
 +\frac{2c_{K,A}}{\beta^2}\varepsilon^2+2d.
 \label{eq:global-reciprocal-split-deficit}
\end{equation}
Indeed $\widehat\kappa\le\beta$ implies
$\widehat\kappa^{-1}-\beta^{-1}
 \ge(\beta-\widehat\kappa)/\beta^2$.
Consequently, along any sequence of centered interior two-pole
selections at fixed $K,A$ with $\Delta_{\rm sp}\to0$, the
selected branch areas satisfy $s_\pm\to A/2$ and the kernel
increments satisfy $d\to0$. This concerns the selected
configuration; it does not assert geometric convergence of
the domains.

These spectral estimates use the moment-centered minimizing profile; the even test profiles above only bound its minimum value. They apply to the selected interior pair. The one-pole alternative instead uses the tail-subtracted physical energy.

\subsection{Uniform coercivity modulo the minimizing profile}\label{sec:uniformkernel}
\begin{proposition}\label{prop:uniformkernel}
Fix $K\in\mathbb R$ and $A>0$, with $KA<4\pi$ if $K>0$, and set $\beta=\beta_K(A/2)$. Then
\begin{equation}
 \gamma_{K,A}:=2-\frac{A\beta}{2c}=2-\frac{A\beta}{4\pi}>0.
 \label{eq:uniform-kernel-gamma}
\end{equation}
In the setting of Proposition~\ref{prop:kernel}, assume $\mu(X)=A$ and $\kappa(h)\le\beta$. Let $P_h$ be the $\mathcal H$-orthogonal projection onto $\Span\{u_h\}$. Since Proposition~\ref{prop:kernel} gives
$\|u_h\|_{\mathcal H}^2=\mathcal N(u_h)=\kappa(h)$,
\[
 P_hv=\kappa(h)^{-1}\langle v,u_h\rangle_{\mathcal H}u_h.
\]
For every real $v\in\mathcal H$,
\begin{equation}
 \mathcal N(v)-\kappa(h)D_h(v)
 \ge\gamma_{K,A}\mathcal N(v-P_hv).
 \label{eq:uniform-kernel-remainder}
\end{equation}
In particular, the same constant applies to every domain in Theorem~\ref{thm:main} with these $K,A$ and every closed-bidisk parameter $h=H_\xi$ of Section~\ref{sec:center}, as well as every positive model split of total area $A$.
\end{proposition}
\begin{proof}
\emph{Trace and the second eigenvalue.} For an orthonormal basis $(\zeta_n)$ of $\mathcal H$, Parseval for the evaluation vector gives
$\sum_n|\zeta_n(t)|^2=\|R_c(\cdot,t)\|_{\mathcal H}^2=R_c(t,t)=1/(2c)$. Hence Tonelli yields
\begin{equation}
 \|J_h\|_{\rm HS}^2
 =\int_X\sum_n|\zeta_n(h(x))|^2\,d\mu(x)
 =\frac{A}{2c}
 =\operatorname{tr}(J_h^*J_h)=\operatorname{tr}\mathcal T_h.
 \label{eq:kernel-trace}
\end{equation}
Thus $J_h^*J_h$ and $\mathcal T_h=J_hJ_h^*$ are nonnegative trace-class operators; see also \cite[Section 6.3]{Teschl}. Their nonzero spectra, including multiplicities, coincide. Proposition~\ref{prop:kernel} gives
$J_h^*J_h u_h=\kappa(h)^{-1}u_h$, so the largest eigenvalue of $J_h^*J_h$ is
$\Lambda_1=1/\kappa(h)$ and its eigenspace is $\Span\{u_h\}$. If $\Lambda_2$ denotes the next eigenvalue (and $\Lambda_2=0$ in rank one), then
$\Lambda_2\le A/(2c)-\Lambda_1$. The spectral theorem on the orthogonal complement of $u_h$, including the nullspace of $J_h^*J_h$, yields
\begin{equation}
 \mathcal N(v)-\kappa(h)D_h(v)
 \ge\left(1-\frac{\Lambda_2}{\Lambda_1}\right)\mathcal N(v-P_hv)
 \ge\left(2-\frac{A\kappa(h)}{2c}\right)\mathcal N(v-P_hv).
 \label{eq:kernel-trace-coercivity}
\end{equation}
\emph{Positivity of the coefficient.} It remains to prove $A\beta<8\pi$, which is equivalent to positivity in \eqref{eq:uniform-kernel-gamma}. Put $m=A/2$. For $K\le0$, test the model with $f(r)\cos\theta$, $f=r-r^3/3$, in conformal radius. Lemma~\ref{lem:cumulative}, applied to $q_{K,m}$ with comparison curvature zero, increases its mass relative to the flat density $m/\pi$. Since
$\int_0^1(rf'^2+f^2/r)\,dr=14/27$ and $\int_0^1rf^2\,dr=11/72$, one obtains $A\beta\le224\pi/33<8\pi$.
For $K>0$, let $R$ be the half-area radius and $x=\cos(\sqrt K R)=1-KA/(4\pi)\in(0,1)$. The trial $S_K(r)\cos\theta$ has quotient
$K(x^2+x+4)/[(1-x)(x+2)]$, by direct polar integration. Hence
\begin{equation}
 A\beta\le4\pi\frac{x^2+x+4}{x+2}
 =4\pi\left(2-\frac{x(1-x)}{x+2}\right)<8\pi.
 \label{eq:spherical-model-upper-trial}
\end{equation}
\emph{Domain and model parameters.} Since $c=2\pi$ and $\kappa(h)\le\beta$, \eqref{eq:kernel-trace-coercivity} proves \eqref{eq:uniform-kernel-remainder}. For parameters arising from the domain this hypothesis follows from Theorem~\ref{thm:transmission} and continuity in Lemma~\ref{lem:continuity}, including coalescing-pole parameters and parameters on the boundary of the bidisk; for model splits it follows from Proposition~\ref{prop:modelsplit}.
\end{proof}
The same calculation gives the explicit lower bounds
\begin{equation}
 \gamma_{K,A}\ge
 \begin{cases}
 10/33,& K\le0,\\[2pt]
 x(1-x)/(x+2),&K>0,\quad x=1-KA/(4\pi)\in(0,1).
 \end{cases}
 \label{eq:uniform-kernel-explicit-lower}
\end{equation}
These explicit bounds need not be optimal; the spherical bound is not uniformly positive as $KA\uparrow4\pi$.
The remainder controls $\mathcal N(v-P_hv)$ relative to the datum-dependent subspace $\Span\{u_h\}$, also at boundary parameters. The energy is full-line, not tail-subtracted; this is not a physical Neumann spectral-gap estimate.

\paragraph{Optimality in the abstract measure class.}\label{par:abstract-sharpness}
For fixed admissible $K,A$, the constant $\gamma_{K,A}$ itself is the largest one valid for every abstract datum in Proposition~\ref{prop:uniformkernel}. First,
\begin{equation}
 \frac{2c}{A}<\beta<\frac{4c}{A}.
 \label{eq:abstract-beta-window}
\end{equation}
The upper bound was proved above. For the lower bound, the equal-area model time measure has density $B_0(|t|)>0$. Two smooth profiles with disjoint compact supports have linearly independent images under $J_h$, so $J_h^*J_h$ has rank at least two. Its largest eigenvalue $1/\beta$ is therefore strictly smaller than its trace $A/(2c)$, proving the lower bound.
Set
\[
 q=\frac{4c}{A\beta}-1\in(0,1),\qquad
 L=-\frac{\log q}{2c}>0,\qquad
 \rho_L=\frac A2(\delta_{-L}+\delta_L).
\]
This is an admissible abstract datum by taking $X=\mathbb R$, $\mu=\rho_L$ and $h(t)=t$.
In the orthonormal point-mass basis of $L^2(\rho_L)$, its kernel matrix and nonzero eigenvalues are
\begin{equation}
 \frac A{4c}\begin{pmatrix}1&q\\q&1\end{pmatrix},\qquad
 \Lambda_\pm=\frac A{4c}(1\pm q),\qquad \kappa=\Lambda_+^{-1}=\beta.
 \label{eq:two-atom-sharp-kernel}
\end{equation}
The positive minimizing profile is even. The odd function
$v(t)=R_c(t,-L)-R_c(t,L)$ is energy-orthogonal to it, so $P_hv=0$.
Using the reproducing identity and evaluating at the two atoms gives
\[
 \mathcal N(v)=\frac{1-q}{c},\qquad
 D_h(v)=\frac{A(1-q)^2}{4c^2}.
\]
Consequently
\begin{equation}
 \frac{\mathcal N(v)-\beta D_h(v)}{\mathcal N(v)}
 =1-\frac{\Lambda_-}{\Lambda_+}
 =2-\frac{A\beta}{2c}=\gamma_{K,A}.
 \label{eq:abstract-gamma-equality}
\end{equation}
The optimality statement is for the abstract measure class; no geometric realization of this two-atom datum, or optimality in the geometric subclasses, is asserted.
The bound $\kappa(h)\le\beta$ cannot be replaced by total mass alone: if instead $L\to\infty$, then $\kappa_L\to4c/A>\beta$ and
$1-\Lambda_-/\Lambda_+\to0$.

\FloatBarrier
\appendix
\section{Domains with finite cyclic symmetry}\label{sec:cyclic}
Return to either setting of Theorem~\ref{thm:main}. Suppose some conformal disk coordinate has density
\begin{equation}
 \rho(e^{2\pi i/N}z)=\rho(z),\qquad N\in\mathbb N,\ N\ge3.
 \label{eq:cyclic-symmetry}
\end{equation}
The pullback density arising from setting \textup{(L)} or \textup{(G)} is positive
and continuous in the interior and has finite mass. No boundary bounds or
extension of this density to $\partial\D$ are assumed; its curvature bound is
understood in the sense of \eqref{eq:LL-curvature}. For a nonzero nonnegative nondecreasing absolutely continuous $f$ with $f(r)=O(r)$ and finite angular energy, the first and second angular moments vanish:
\[
 \int f(r)e^{i\theta}\rho=0,\qquad
 \int f(r)^2e^{2i\theta}\rho=0.
\]
Indeed their rotation multipliers are not one. The real components $f\cos\theta,f\sin\theta$ are mean-zero with scalar mass and energy matrices, and belong to $H^1(\Omega)$. Their positive mass gives an independent pair; adjoining the constant yields the three-dimensional min--max space for $\lambda_2$. Lemma~\ref{lem:cumulative} gives
\begin{equation}
 \lambda_2(\Omega,g)\le
 \frac{2\pi\int_0^1(rf'^2+f^2/r)\,dr}{\int_\D f^2\rho}
 \le\frac{2\pi\int_0^1(rf'^2+f^2/r)\,dr}{\int_\D f^2q_{K,A}}.
 \label{eq:cyclic-trial}
\end{equation}
Order two symmetry alone does not eliminate the second angular moment.

\begin{theorem}\label{thm:cyclic}
Assume \eqref{eq:cyclic-symmetry}, $K_g\le K$, $K>0$ and $0<KA<4\pi$. If $KA\le2\pi$, then
\begin{equation}
 \lambda_2(\Omega,g)\le\beta_K(A)<\beta_K(A/2).
 \label{eq:cyclic-fullarea}
\end{equation}
If $2\pi<KA<4\pi$ and $x=KA/(2\pi)-1$, then
\begin{equation}
 \lambda_2(\Omega,g)/K\le B(x):=
 \frac{4/3+\operatorname{atanh}x}{2/3+x}.
 \label{eq:B-cyclic}
\end{equation}
For $0<x\le x_*$, where $\operatorname{atanh}x_*=2x_*$ is the unique nonzero root in $(0,1)$,
\begin{equation}
 \lambda_2(\Omega,g)\le2K<\beta_K(A/2).
 \label{eq:cyclic-2K}
\end{equation}
In particular the stronger bounds hold through $KA=3.9\pi$.
\end{theorem}
\begin{proof}
The radial factors below are pulled back from geodesic polar coordinates
to the conformal unit disk by an increasing radial change of coordinate.
At or below a hemisphere, the proof of Lemma~\ref{lem:angularmodel} gives an increasing angular factor; at a hemisphere it is $\sin(\sqrt K r)$ with eigenvalue $2K$. Its model eigenvalue decreases strictly with radius on this interval. At a smaller radius $R_1$, put $\beta_1=\beta_K(V_K(R_1))$. The positive angular factor satisfies
\[
 y=S_Kv',\qquad y'=(1-\beta_1S_K^2)v/S_K,
 \qquad y(0)=y(R_1)=0.
\]
If $\beta_1\le S_K(R_1)^{-2}$, strict increase of $S_K$ in the interior of this hemisphere interval would give $y'>0$ on $(0,R_1)$, contradicting the endpoint values. Thus $\beta_1>S_K(R_1)^{-2}$. Extend $v$ constantly to $R_2>R_1$. The annular quotient is
$\int_{R_1}^{R_2}S_K^{-1}/\int_{R_1}^{R_2}S_K\le S_K(R_1)^{-2}$,
strictly below the inner quotient. Min--max in the first angular sector proves the claim. Using the full-area model factor in \eqref{eq:cyclic-trial} now proves \eqref{eq:cyclic-fullarea}, including the hemisphere endpoint.

Above a hemisphere, let the full-area boundary have polar angle $\Theta\in(\pi/2,\pi)$, so $x=-\cos\Theta$. Use
$v(\theta)=\sin\theta$ for $\theta\le\pi/2$ and $v=1$ thereafter. It is $C^1$ and nondecreasing. For one real component its model energy and mass are
\begin{align}
 E&=\pi\left[\int_0^{\pi/2}(\cos^2\theta\sin\theta+\sin\theta)\,d\theta
 +\int_{\pi/2}^{\Theta}\frac{d\theta}{\sin\theta}\right]
 =\pi(4/3+\operatorname{atanh}x),\notag\\
 M&=\frac\pi K\left[\int_0^{\pi/2}\sin^3\theta\,d\theta
 +\int_{\pi/2}^{\Theta}\sin\theta\,d\theta\right]
 =\frac\pi K(2/3+x).
 \label{eq:cyclic-integrals}
\end{align}
Here
$\int_{\pi/2}^{\Theta}\csc\theta\,d\theta
=\log\tan(\Theta/2)=\operatorname{atanh}x$ because
$\tan^2(\Theta/2)=(1+x)/(1-x)$, and
$\int_{\pi/2}^{\Theta}\sin\theta\,d\theta=x$.
Their ratio proves \eqref{eq:B-cyclic}. Now $B(x)\le2$ exactly when
$h(x):=\operatorname{atanh}x-2x\le0$.
Since $h'(x)=(2x^2-1)/(1-x^2)$ and $h(1-)=+\infty$, there is exactly one nonzero root, with $h\le0$ before it. The half-area model satisfies $\beta>2K$. Finally,
$e^{19/5}>\sum_{j=0}^6(19/5)^j/j!=457179601/11250000>39$,
so $\operatorname{atanh}(19/20)<19/10$. Thus $3.9\pi$ lies in this range.
\end{proof}
For $2\pi<KA\le2\pi(1+x_*)$, the stronger estimate is $\lambda_2\le2K$, with $x_*=0.957504024077\ldots$.

\subsection{A nonradial example beyond a hemisphere}

On $\D$, take
\begin{equation}
 \rho_\varepsilon(z)=\frac{12e^{2\varepsilon\Re(z^3)}}{(1+3|z|^2)^2},
 \qquad 0<\varepsilon\le1/100.
 \label{eq:threefold-example}
\end{equation}
The baseline has curvature $1$ and area $3\pi$. Harmonicity of $\Re(z^3)$ gives
\begin{equation}
 K_{g_\varepsilon}=e^{-2\varepsilon\Re(z^3)},\qquad
 K_\varepsilon=\sup K_{g_\varepsilon}=e^{2\varepsilon}.
 \label{eq:threefold-curvature}
\end{equation}
Strict Jensen on each angular circle and the pointwise exponential bound give
$3\pi<A_\varepsilon\le3\pi e^{2\varepsilon}$. Since $e^t\le(1-t)^{-1}$ for $0\le t<1$,
\[
 3\pi<K_\varepsilon A_\varepsilon\le3\pi e^{4\varepsilon}
 \le\frac{3\pi}{1-4\varepsilon}\le\frac{25\pi}{8}<3.9\pi.
\]
The metric is nonradial and threefold symmetric; Theorem~\ref{thm:cyclic} gives
$\lambda_2\le2K_\varepsilon<\beta_{K_\varepsilon}(A_\varepsilon/2)$.

\section{Interface regularity of the Green-time profile}\label{app:regularity}

\subsection{A smooth disk with a minimizing profile that is not $C^2$}\label{app:log-example}
Smoothness of the metric does not imply boundedness of the time density at the interface. Work on the flat unit disk, fix $0<r<1$, and choose poles $a=-r$, $b=r$ in \eqref{eq:green}. Its positive branch is the left half-disk. The zero interface meets the outer circle at $z=\pm i$.
For a local branch $L=\log\Xi$ near $i$, direct differentiation gives
\[
 L'(i)=0,\qquad L''(i)=\frac{4ir(1-r^2)}{(1+r^2)^2}.
\]
Use $z=-s+i(\sqrt{1-s^2}-v)$, with $s,v>0$ small. These coordinates have area Jacobian one and flatten the interface and outer boundary. Since $H$ vanishes on both axes, smooth factorization gives $H(s,v)=sv\,g(s,v)$. To identify the leading coefficient, note that at $(s,v)=(0,0)$ one has $z=i$, $z_s=-1$, $z_v=-i$, and $L'(i)=0$. Hence
\[
 \partial_s\partial_vH(0,0)
 =-c^{-1}\Re\bigl(L''(i)z_sz_v\bigr)
 =\frac{4r(1-r^2)}{c(1+r^2)^2}.
\]
Therefore
\begin{equation}
 H(s,v)=sv\,g(s,v),\qquad
 g(s,v)=d_r+O(s+v),\qquad
 d_r=\frac{4r(1-r^2)}{c(1+r^2)^2}>0.
 \label{eq:corner-time-factorization}
\end{equation}
Put $w=vg(s,v)$, so $H=sw$. In a sufficiently small corner neighborhood this is a smooth change of variables, with
$\partial v/\partial w=d_r^{-1}+O(s+w)$.
Writing $dA=J(s,w)\,ds\,dw$, with $J(s,w)=d_r^{-1}+O(s+w)$, the substitution $t=sw$ gives the local time-density element
\[
 J(s,t/s)\frac{ds}{s}
 =\left(d_r^{-1}+O(s+t/s)\right)\frac{ds}{s}.
\]
On a fixed small coordinate rectangle, $s$ ranges from a constant multiple of $t$ to a fixed positive endpoint. This corner contributes $d_r^{-1}\log(1/t)+O(1)$: both the integrated error $O(1+t/s^2)$ and changes in the fixed endpoint constants are bounded. The lower corner contributes the same by conjugation, and the regular part of the zero level contributes a bounded term. Hence, for fixed $r$,
\begin{equation}
 B_+(t)=C_r\log(1/t)+O(1),\qquad
 C_r=\frac2{d_r}=\frac{\pi(1+r^2)^2}{r(1-r^2)}
 \quad(t\downarrow0).
 \label{eq:log-time-density}
\end{equation}
The measure is nevertheless absolutely continuous and finite, because the logarithm is integrable.
Reflection exchanges the branches, so uniqueness makes the positive minimizing profile $u$ normalized by $u(0)=1$ even. The regularity statement following \eqref{eq:atomic-interface} gives $u'(0)=0$ and $u(t)=1+o(t)$. For $t>0$, \eqref{eq:measureode} and \eqref{eq:log-time-density} imply
\[
 u''(t)=-\kappa(H)C_r\log(1/t)+O(1).
\]
Integrating twice from zero gives
\begin{equation}
 u(t)=1-\frac{\kappa(H)C_r}{2}t^2\log(1/|t|)+O(t^2).
 \label{eq:log-profile-expansion}
\end{equation}
Thus $u\notin C^2$ and $u\notin W^{2,\infty}$ near zero, while $u\in W^{2,1}_{\rm loc}$.
This concerns the minimizing profile. Under the hypothetical coefficient
equalities \eqref{eq:doubleequality}, the proof of Theorem~\ref{thm:transmission}
derives the singular interface term in \eqref{eq:criticalarea}.

 \noindent\textbf{Acknowledgement:}  M. Liu is supported by Zhejiang Provincial Natural Science Foundation of China (LQN25A010007) and the Fundamental Research
Funds for the Provincial Universities of Zhejiang (GK259909299001-029). W.
Zou is funded by National Key R\&D Program of China (Grant 2023YFA1010001) and NSFC (12171265).

\vskip0.2in
 
\noindent\textbf{Declarations.} \quad AI tools   were used during the preparation of the paper.  The authors assume full responsibility for the
 mathematical validity, accuracy, and integrity of the proofs presented in the manuscript.

\vskip0.12in
\smallskip
\noindent\textbf{Conflict of interest.}\quad On behalf of all authors, the corresponding author states that there is no conflict of interest.

\vskip0.12in
\smallskip
\noindent\textbf{Data availability statement.}\quad All data generated or analyzed during this study are included in this article.

\end{document}